\documentclass[11pt, reqno, twoside]{amsart}
\usepackage{amscd}
\usepackage{amsfonts}
\usepackage{amsmath}
\usepackage{amssymb}
\usepackage{amsthm}
\usepackage{fancyhdr}
\usepackage{latexsym}
\usepackage[colorlinks=true, pdfstartview=FitV, linkcolor=blue, citecolor=blue, urlcolor=blue]{hyperref}
\usepackage{enumitem}      
\usepackage{mathtools}            
\usepackage{upgreek}
\usepackage{mathrsfs}  
\usepackage{indentfirst} 
\usepackage{color}
\usepackage{tikz}
\usetikzlibrary{arrows.meta}
\usetikzlibrary{decorations.markings}
\tikzset{->-/.style={decoration={
  markings,
  mark=at position #1 with {\arrow{>}}},postaction={decorate}}}
  \tikzset{middlearrow/.style={
        decoration={markings,
            mark= at position 0.55 with {\arrow{#1}} ,
        },
        postaction={decorate}
    }
}
\usepackage{caption}          
\usepackage[normalem]{ulem}
\newcommand{\nn}{\nonumber}
\newcommand{\p}{\partial}

\usepackage{accents}
\renewcommand*{\dot}[1]{%
\accentset{\mbox{\large\bfseries .}}{#1}}
\usepackage{thmtools}
\declaretheoremstyle[headfont=\kpfonts]{normalhead}

\newtheorem{theorem}{Theorem}[section]

\theoremstyle{definition}

\makeatletter
\def\mathcolor#1#{\@mathcolor{#1}}
\def\@mathcolor#1#2#3{%
\protect\leavevmode
\begingroup
\color#1{#2}#3%
\endgroup
}
\makeatother
\numberwithin{figure}{section}
\numberwithin{equation}{section}
\usepackage{geometry}
\newcommand{\appendarrow}[3]{
\tikzset{
    middlearrow/.style args={#1}{
        decoration={
            markings,
            mark= at position 0.5 with
                {
                    \coordinate (exy1) at (#1/-6.0,0pt);
                    \coordinate (exy2) at (-0.5*#1,#1/3.0);  
                    \coordinate (exy3) at (0.5*#1,0pt);
                    \coordinate (exy4) at (-0.5*#1,#1/-3.0); 
                },
        },
        postaction=decorate
    },
    middlearrow/.default= 3pt
}
\draw[middlearrow=#1] #3;
\fill[#2] (exy1) -- (exy2) -- (exy3) -- (exy4) -- cycle;
}
\makeatletter
\@namedef{subjclassname@2020}{%
  \textup{2020} Mathematics Subject Classification}
\makeatother
\begin{document}
\title{
The Robin and Neumann Problems for the Nonlinear Schr\"odinger Equation on the Half-Plane 
}
\author{A. Alexandrou Himonas \& Dionyssios Mantzavinos$^*$}
\begin{abstract}
This work studies the initial-boundary value problem 
of the two-dimensional nonlinear Schr\"odinger equation 
on the half-plane with initial data in Sobolev spaces 
and Neumann or Robin boundary data in appropriate 
Bourgain spaces.
It establishes well-posedness in the sense of Hadamard 
by utilizing the explicit solution formula  for the forced linear  initial-boundary value problem obtained via Fokas's unified transform,  and a contraction mapping argument.
\end{abstract}
\date{April 27, 2022. $^*$\!\textit{Corresponding author}: mantzavinos@ku.edu}
\maketitle
\markboth
{The Robin and Neumann problems for the nonlinear Schr\"odinger equation on the half-plane}
{A. Himonas \& D. Mantzavinos}
%

%
%
%
%
%
%
%
%
%
\section{Introduction}
\label{2d-rnls-intro}

We study the initial-boundary value problem (ibvp) of the nonlinear Schr\"odinger (NLS) equation on the half-plane with a Robin boundary condition, that is
\begin{equation}
\label{2d-rnls-ibvp}
\begin{aligned}
&i u_t + u_{x_1x_1}+u_{x_2x_2} = \pm |u|^{\alpha-1} u,
\quad (x_1, x_2)\in \mathbb R \times \mathbb R^+, \ t \in (0, T),  
\\
&u(x_1, x_2, 0) = u_0(x_1, x_2),
\\
&(u_{x_2} + \gamma u) (x_1, 0, t)  = g(x_1, t),
\end{aligned}
\end{equation}
where $(\alpha -1)/2 \in \mathbb{N}$ and $\gamma \in \mathbb R$.
 When $\gamma=0$, this is the Neumann problem, which we  examine at the end of this work.
Here, we establish the local well-posedness of ibvp \eqref{2d-rnls-ibvp} for initial data $u_0$ in the 
Sobolev space of the half-plane
$H^s(\mathbb R_{x_1}\times \mathbb R_{x_2}^+)$
and boundary data in the Bourgain-type space $B_T^s$
suggested by the solution estimate of the reduced pure linear ibvp (see  \eqref{2d-rnls-pure-se}). 
We recall that
 $H^s(\mathbb R_{x_1}\times \mathbb R_{x_2}^+)$ is defined as the restriction of the Sobolev space $H^s(\mathbb R^2)$ 
 to $\mathbb R_{x_1}\times \mathbb R_{x_2}^+$  with~norm
\begin{equation}
\left\| f \right\|_{H^s(\mathbb R_{x_1}\times \mathbb R_{x_2}^+)}
=
\inf \left\{ \left\| F \right\|_{H^s(\mathbb R^2)}: F \in H^s(\mathbb R^2) \text{ and } F|_{\mathbb R_{x_1}\times \mathbb R_{x_2}^+} = f \right\}.
\end{equation}
The boundary data space $B^s_T$, which can be thought as expressing the time regularity of the linear homogeneous problem in  two dimensions,
is defined by 
\begin{equation}\label{bst-def}
B_T^s 
= 
X_T^{0, \frac{2s-1}{4}} \cap  X_T^{s, -\frac 14},
\end{equation}
where  $X_T^{0, \frac{2s-1}{4}}$ and $X_T^{s, -\frac 14}$ are 
the Bourgain-type spaces defined via the norm
\begin{equation}
\label{xsbt-def}
\left\|g\right\|_{X_T^{\sigma, b}}^2
=
\int_{k_1\in\mathbb R} 
\big(1+k_1^2\big)^\sigma 
\big\|e^{ik_1^2t} \, \widehat g^{x_1}(k_1, t)\big\|_{H^b(0, T)}^2 dk_1
\end{equation}
with $\widehat g^{x_1}$ denoting the Fourier transform of $g$ with respect to $x_1$, i.e.
\begin{equation}
\widehat g^{x_1}(k_1, t)
=
\int_{x_1\in\mathbb R} e^{-ik_1x_1} g(x_1, t) dx_1, \quad k_1 \in \mathbb R, \ t \in (0, T).
\end{equation}
Furthermore, we note that the  spaces $X_T^{\sigma, b}$ can be regarded as restrictions on $\mathbb R \times (0, T)$ of the celebrated Bourgain spaces   $X^{\sigma, b}(\mathbb R_x\times \mathbb R_t)$, which are defined via the norm  \cite{b1993}
\begin{equation}\label{xsb-def}
\begin{aligned}
\left\|g\right\|_{X^{\sigma, b}}^2
&=
\int_{k_1\in\mathbb R} 
\big(1+k_1^2\big)^\sigma 
\left(1+\big|\tau+k_1^2\big|^2\right)^b
\left|\widehat g(k_1, \tau)\right|^2 dk_1
\\
&=
\int_{k_1\in\mathbb R} 
\big(1+k_1^2\big)^\sigma 
\left\|e^{ik_1^2t} \, \widehat g^{x_1}(k_1, t)\right\|_{H^b(\mathbb R_t)}^2 dk_1.
\end{aligned}
\end{equation}

Now, we are able to state the main result of this work 
more precisely as follows.
\begin{theorem}[\textcolor{blue}{Local well-posedness}]
\label{2d-rnls-t}
Suppose $1<s<\frac 32$. Then, for initial data  $u_0\in H^s(\mathbb R_{x_1}\times \mathbb R_{x_2}^+)$ and Robin ($\gamma\neq 0$) or Neumann ($\gamma=0$) boundary data $g\in B_T^s$, the NLS  ibvp~\eqref{2d-rnls-ibvp}   is locally well-posed in the sense of Hadamard. More precisely, for  
\begin{equation}\label{lifespan}
T^* = \min\big\{
T, \, c_{s, \gamma, \alpha}  \big( \left\|u_0\right\|_{H^s(\mathbb R_{x_1}\times \mathbb R_{x_2}^+)} + \left\|g\right\|_{B_T^s} \big)^{-2(\alpha-1)}
\big\}, \quad c_{s, \gamma, \alpha} >0, 
\end{equation}
there exists a unique solution $u\in C([0, T^*]; H^s(\mathbb R_{x_1}\times \mathbb R_{x_2}^+))$ which satisfies the estimate
\begin{equation}
\sup_{t\in [0, T^*]} \left\|u(t)\right\|_{H^s(\mathbb R_{x_1}\times \mathbb R_{x_2}^+)} 
\leqslant
2c_{s, \gamma} \, \big(\left\|u_0\right\|_{H^s(\mathbb R_{x_1}\times \mathbb R_{x_2}^+)} + \left\|g\right\|_{B_T^s}\big), \quad c_{s, \gamma} >0.
\end{equation}
In addition, the data-to-solution map $\left\{u_0, g\right\}\mapsto u$ is locally Lipschitz continuous.
\end{theorem}

Theorem \ref{2d-rnls-t} completes the picture of the NLS ibvp
well-posedness on the half-plane for smooth data initiated 
in our earlier work \cite{hm2020} for the case of Dirichlet data. 
In one spatial dimension and on the half-line, the well-posedness of the NLS ibvp by utilizing the Fokas unified transform method was established in \cite{fhm2017}  for Dirichlet data and in \cite{hmy2019nls} and \cite{hm2021} for Neumann and Robin data respectively.
Furthermore, concerning the ibvp for the Korteweg-de Vries (KdV) equation on the half-line, this approach for proving 
well-posedness has been implemented   in \cite{fhm2016} for Dirichlet data and in \cite{hmy2021} for Neumann and Robin data.
We also mention that there are two other approaches  in the literature for studying the well-posedness of the ibvp for KdV and NLS on the half-line with Dirichlet data. 
The first approach is due to Bona, Sun and Zhang \cite{bsz2002,bsz2006,bsz2018} and uses the Laplace transform in the temporal variable for analyzing the linear problem (see also \cite{et2016}).
The second approach was developed by Colliander and Kenig for the generalized KdV on the half-line \cite{ck2002}, and later on by Holmer for the NLS and KdV on the half-line \cite{h2005,h2006}, and is based on expressing the linear ibvp as a superposition of initial value problems.
Moreover, besides~\cite{hm2020}, two other works in the literature on the well-posedness of the two-dimensional NLS equation on the half-plane are those by  Ran, Sun and Zhang \cite{rsz2018} and Audiard \cite{a2018}.

The well-posedness of the  nonlinear ibvp \eqref{2d-rnls-ibvp} will be established via a contraction mapping argument utilizing the unified transform solution formula and the estimates obtained for the forced linear version of that ibvp.
Therefore, the first step of our approach is to derive the 
Fokas unified transform solution  for the  forced linear ibvp 
\begin{equation}\label{2d-rfls-ibvp}
\begin{aligned}
&i u_t + u_{x_1x_1}+u_{x_2x_2} = 
f(x_1, x_2, t) \in C([0, T]; H^s(\mathbb R_{x_1}\times \mathbb R_{x_2}^+)),
\\
&u(x_1, x_2, 0) = u_0(x_1, x_2)\in H^s(\mathbb R_{x_1}\times \mathbb R_{x_2}^+),  
\\
&(u_{x_2} + \gamma u)(x_1, 0, t) = g(x_1, t)\in B_T^s.
\end{aligned}
\end{equation}
This formula is given by (see Section \ref{2d-rnls-utm-s} for an outline of its derivation)
\begin{align}\label{2d-rfls-sol-T}
&\quad
u(x_1, x_2, t)
=
S\big[u_0, g; f\big](x_1, x_2, t)
\\
&=
\frac{1}{(2\pi)^2} \int_{k_1\in\mathbb R} \int_{k_2\in\mathbb R} e^{ik_1x_1+ik_2x_2-i(k_1^2+k_2^2)t} \, \widehat u_0(k_1, k_2) dk_2 dk_1
\nn\\
&
\quad
+\frac{1}{(2\pi)^2} \int_{k_1\in\mathbb R} \int_{k_2\in\mathcal C} e^{ik_1x_1+ik_2x_2-i(k_1^2+k_2^2)t} \, \frac{k_2+i\gamma}{k_2-i\gamma} \, \widehat u_0(k_1,-k_2) dk_2 dk_1
\nn\\
&\quad
-\frac{i}{(2\pi)^2} \int_{k_1\in\mathbb R} \int_{k_2\in\mathbb R} e^{ik_1x_1+ik_2x_2-i(k_1^2+k_2^2)t} \int_{t'=0}^t  e^{i(k_1^2+k_2^2)t'}\widehat f(k_1, k_2, t')dt' dk_2 dk_1
\nn\\
&
\quad
-\frac{i}{(2\pi)^2}
\int_{k_1\in\mathbb R} \int_{k_2\in\mathcal C}
e^{ik_1x_1+ik_2x_2-i(k_1^2+k_2^2)t}  \, \frac{k_2+i\gamma}{k_2-i\gamma} \,
\int_{t'=0}^t 
e^{i(k_1^2+k_2^2)t'}\widehat f(k_1, -k_2, t')dt'
dk_2 dk_1
\nn\\
&\quad
-\frac{i}{(2\pi)^2} \int_{k_1\in\mathbb R}\int_{k_2\in\mathcal C} e^{ik_1x_1+ik_2x_2-i(k_1^2+k_2^2)t}
 \, \frac{2k_2}{k_2-i\gamma} \, \widetilde g (k_1, k_1^2+k_2^2, T)
dk_2dk_1,
\nn
\end{align}
where the complex contour  $\mathcal C$ is either $\p D$ (for $\gamma\leqslant 0$) or $\p \widetilde D$ (for $\gamma>0$), as shown in Figure~\ref{2d-rnls-dplus-f},

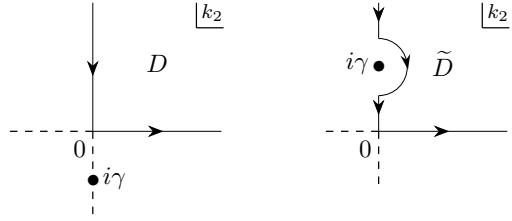
\begin{figure}[ht!]
\centering
\vspace{2cm}
\hspace{-5cm}
\begin{tikzpicture}[scale=0.95]
\pgflowlevelsynccm
\draw[line width=.5pt, black, dashed](0,0)--(-1.2,0);
\draw[line width=.5pt, black, dashed](0,0)--(0,-1.2);
\draw[line width=.5pt, black](1.45,1.8)--(1.45,1.45);
\draw[line width=.5pt, black](1.45,1.45)--(1.83,1.45);
\node[] at (1.605, 1.66) {\fontsize{8}{8} $k_2$};
\node[] at (-0.25, -0.25) {\fontsize{10}{10} $0$};
\draw[middlearrow={Stealth[scale=1.3, reversed]}] (0,0) -- (90:1.8);
\draw[middlearrow={Stealth[scale=1.3]}] (0,0) -- (0:1.8);
\node[] at (0.9, 0.95) {\fontsize{10}{10}\it $D$};
\node[] at (0.3, -0.65) {\fontsize{10}{10}\it $i\gamma$};
\node[dotted] at (0.007, -0.7) {\textbullet};
\draw[xshift=4cm, line width=.5pt, black, dashed](0,0)--(-0.8,0);
\draw[xshift=4cm, line width=.5pt, black, dashed](0,0)--(0,-0.8);
\draw[xshift=4cm, line width=.5pt, black](1.45,1.8)--(1.45,1.45);
\draw[xshift=4cm, line width=.5pt, black](1.45,1.45)--(1.83,1.45);
\node[xshift=4cm] at (1.605, 1.66) {\fontsize{8}{8} $k_2$};
\node[xshift=4cm] at (-0.25, -0.25) {\fontsize{10}{10} $0$};
\draw[xshift=4cm, middlearrow={Stealth[scale=1.3]}] (0,1.8) -- (90:1.3);
\draw[xshift=4cm, middlearrow={Stealth[scale=1.3]}] (0,0.5) -- (90:0);
\draw[xshift=4cm, middlearrow={Stealth[scale=1.3]}] (0,0) -- (0:1.8);
\node[xshift=4cm] at (0.9, 0.95) {\fontsize{10}{10}\it $\widetilde D$};
\node[xshift=4cm] at (-0.3, 0.95) {\fontsize{10}{10}\it $i\gamma$};
\node[xshift=4cm, dotted] at (0.007, 0.9) {\textbullet};
\draw[xshift=4cm] (0, 0.5) arc (-90:90:0.4);
\appendarrow{6pt}{black}{(4.38, 1) to [out=90, in=90] (4.4, 0.9);}
\end{tikzpicture}
\vspace{12mm}
\caption{The regions $D$  and $\widetilde D$  with their positively oriented boundaries $\p D$ and $\p \widetilde D$, which assume the role of the contour $\mathcal C$ in  formula \eqref{2d-rfls-sol-T} for $\gamma\leqslant 0$ and $\gamma>0$ respectively.}
\label{2d-rnls-dplus-f}
\end{figure}

\noindent
the terms $\widehat u_0$ and $\widehat f$  denote the \textbf{half-plane} Fourier transforms of $u_0$ and $f$ defined according to
\begin{equation}\label{ft-hp-def}
\widehat \varphi(k_1, k_2) = \int_{x_1\in \mathbb R} \int_{x_2=0}^\infty e^{-ik_1x_1-ik_2x_2} \, \varphi(x_1, x_2) dx_2 dx_1,
\end{equation}
the transform $\widetilde g$ is defined in terms of the boundary data $g$ by
\begin{equation}\label{gtil-def}
\widetilde g (k_1, k_1^2+k_2^2, T)
= \int_{t=0}^T e^{i(k_1^2+k_2^2)t} \int_{x_1\in \mathbb R} e^{-ik_1x_1} g(x_1, t) dx_1 dt.
\end{equation}

The second step of our approach consists in estimating 
the solution \eqref{2d-rfls-sol-T} of the forced linear ibvp~\eqref{2d-rfls-ibvp}. The estimate derived  is described in the following result.
\begin{theorem}[\textcolor{blue}{Linear ibvp}]
\label{2d-rfls-t}
Suppose $1\leqslant s  <  \frac 32$. Then, Fokas's unified transform formula~\eqref{2d-rfls-sol-T} defines a solution $u = S\big[u_0, g; f\big]$ to the forced linear Schr\"odinger ibvp \eqref{2d-rfls-ibvp} with initial data  $u_0\in H^s(\mathbb R_{x_1}\times \mathbb R_{x_2}^+)$, Robin ($\gamma\neq 0$)  or  Neumann ($\gamma=0$) boundary data $g \in B_T^s$ and forcing  $f\in C([0, T]; H^s(\mathbb R_{x_1}\times \mathbb R_{x_2}^+))$, which satisfies the estimate
\begin{align}\label{2d-rfls-se-r}
\sup_{t\in [0, T]} \left\|S\big[u_0, g; f\big](t)\right\|_{H^s(\mathbb R_{x_1}\times \mathbb R_{x_2}^+)}
&\leqslant
c_{s, \gamma}
\Big(
\left\|u_0\right\|_{H^s(\mathbb R_{x_1}\times \mathbb R_{x_2}^+)}
+
\left\|g\right\|_{B_T^s}
\nn\\
&\hspace*{1.2cm}
+
\sqrt T
\sup_{t\in [0, T]}
\left\|f(t)\right\|_{H^s(\mathbb R_{x_1}\times \mathbb R_{x_2}^+)}
\Big).
\end{align}
\end{theorem}

The unified transform providing the solution formula  \eqref{2d-rfls-sol-T} for the forced linear ibvp studied in this work was introduced in 1997  by Fokas \cite{f1997} (see also the monograph \cite{f2008}).
The method was originally motivated through an effort to develop an ibvp counterpart for the inverse scattering transform used for studying completely integrable \textit{nonlinear} equations in the initial value problem setting.
However, it was immediately realized that Fokas's transform had significant implications also at the level of \textit{linear} ibvps, in particular, taking into account its applicability to linear evolution equations of arbitrary spatial order and dimension,  formulated with any kind of admissible boundary conditions. In this regard, the unified transform provides the direct, \textit{natural} analogue in the linear ibvp setting of the classical Fourier transform used for solving linear initial value problems.
For additional results on the ibvp of NLS,
KdV and related equations via the Fokas method see, for example, \cite{f2002,f2009,fi2004,fp2005,lf2012,lf2015,fl2012,fl2014,c2018,oy2019,bfo2020,dfl2021} 
as well as the review articles \cite{fs2012,dtv2014}.

The NLS equation has an extensive literature. Concerning its physical significance, it arises as a universal model in mathematical physics, e.g. in nonlinear optics \cite{t1964}, water waves \cite{p1983,css1992},  plasmas \cite{ww1977} and Bose-Einstein condensates \cite{ps2003}. 
Moreover, the cubic NLS in one spatial dimension  is a prime example of a completely integrable system and can be studied via the inverse scattering transform  \cite{zs1972}.
Finally, concerning the well-posedness of the initial value problem for NLS in Sobolev spaces, we refer the reader to \cite{gv1979,t1987,cs1989,cw1990,kpv1991,kpv1993,b1993,b1999,t2005,lp2009,cms2016}
and the references therein.

\vskip 2mm
\noindent
\textit{Organization.}
In Section \ref{2d-rnls-pure-s}, we estimate the solution to the 
reduced pure linear Robin problem, which has 
zero forcing, zero initial data, and boundary data
compactly supported in time. Section~\ref{ivp-s} is devoted to the estimation of the linear Schr\"odinger initial value problem. In Section~\ref{fls-ibvp-s}, we combine the results of the previous two sections to prove Theorem \ref{2d-rfls-t} for the forced linear ibvp \eqref{2d-rfls-ibvp} and, in turn, Theorem~\ref{2d-rnls-t} for the well-posedness of the nonlinear ibvp \eqref{2d-rnls-ibvp}. Section~\ref{neumann-s} provides the modifications required in the proofs in the case of the Neumann problem. Finally, in Section \ref{2d-rnls-utm-s} we give a brief derivation of the Fokas unified transform solution formula \eqref{2d-rfls-sol-T}.

%
%
%
%
%
%
%
%
\section{The reduced pure linear ibvp}
\label{2d-rnls-pure-s}

The basis for proving the nonlinear well-posedness Theorem \ref{2d-rnls-t} is provided by the linear estimate of Theorem \ref{2d-rfls-t}  for the forced linear ibvp \eqref{2d-rfls-ibvp}. 
In order to establish this crucial estimate, we begin our analysis from a simplified version of problem \eqref{2d-rfls-ibvp} which involves zero forcing, zero initial data, and compactly supported in time boundary data. 
We call this problem the \textit{reduced pure linear ibvp}, as its non-boundary components are both zero and, furthermore, its boundary datum is reduced to the class of  functions with compact support in $t$. 

More precisely, for the Robin problem ($\gamma \neq 0$), the reduced pure linear ibvp is given by
\begin{equation}\label{2d-rnls-pure}
\begin{aligned}
&iv_t+v_{x_1x_1}+v_{x_2x_2} = 0, \quad (x_1, x_2)\in \mathbb R \times \mathbb R^+, \ t\in (0, 2),
\\
&v(x_1, x_2, 0)= 0, 
\\
&(v_{x_2} + \gamma v)(x_1, 0, t) = g(x_1, t),
\quad
\text{supp}(g)\subset \mathbb R_{x_1}\times (0, 2),
\end{aligned}
\end{equation}
where $g(x_1, t)$ is a globally defined function with  compact support in $t$.
For the Neumann problem ($\gamma\neq 0$), the analysis of the reduced pure linear ibvp  is provided in Section~\ref{neumann-s}.
In the case of problem~\eqref{2d-rnls-pure}, the Fokas unified transform formula \eqref{2d-rfls-sol-T} simplifies to
\begin{equation}\label{2d-nls-pure-sol}
v(x_1, x_2, t)
=
-\frac{i}{(2\pi)^2}
\int_{k_1\in\mathbb R}\int_{k_2\in\mathcal C}
e^{ik_1x_1+ik_2x_2-i(k_1^2+k_2^2)t}
\, \frac{2k_2}{k_2-i\gamma} \, \widehat g(k_1, -k_1^2-k_2^2)\, dk_2dk_1
\end{equation}
where the transform $\widetilde g$ defined by \eqref{gtil-def} has now been replaced by the  Fourier transform $\widehat g$ of $g$ in $x_1$ and $t$ since,  thanks to the compact support of $g$ in $t$,  
\begin{equation}\label{ghat-def}
\widetilde g(k_1, k_1^2+k_2^2, 2)
=
\int_{x_1\in\mathbb R} \int_{t\in\mathbb R} e^{-ik_1x_1+i(k_1^2+k_2^2) t} g(x_1, t) dt dx_1
=
\widehat g(k_1, -k_1^2-k_2^2).
\end{equation}

Next, we will use the Fokas formula \eqref{2d-nls-pure-sol} in order to estimate the solution of the reduced pure linear ibvp \eqref{2d-rnls-pure} in the Hadamard space $C([0, 2]; H^s(\mathbb R_{x_1}\times \mathbb R_{x_2}^+))$. Through this process, we will discover the correct function space for the boundary datum $g(x_1, t)$. In particular, our analysis will reveal the Bourgain spaces $X^{0, \frac{2s-1}{4}}$ and $X^{s, -\frac 14}$ as the global analogues of the spaces $X_T^{0, \frac{2s-1}{4}}$ and $X_T^{s, -\frac 14}$ given in the introduction for the boundary data of the non-reduced problem \eqref{2d-rfls-ibvp}. The precise statement of this result is as follows.
\begin{theorem}[\textcolor{blue}{Basic linear estimate for the Robin problem}]
\label{2d-rnls-pure-t}
Let $s\geqslant 0$ and $\gamma \neq 0$. Then, the solution $v(x_1, x_2, t)$ of the reduced pure linear ibvp \eqref{2d-rnls-pure}, as given by the Fokas  formula \eqref{2d-nls-pure-sol}, satisfies the Hadamard space estimate
\begin{equation}\label{2d-rnls-pure-se}
\sup_{t\in [0, 2]} \left\| v(t) \right\|_{H^s(\mathbb R_{x_1}\times \mathbb R_{x_2}^+)}
\leqslant
c_{s, \gamma}
\big(
\left\|g\right\|_{X^{0, \frac{2s-1}{4}}}
+
\left\|g\right\|_{X^{s, -\frac 14}}
\big),
\end{equation}
where the Bourgain spaces $X^{\sigma, b}$ are defined by \eqref{xsb-def}.
\end{theorem}

In the remaining of this section, we prove Theorem \ref{2d-rnls-pure-t}. We start from the case $\gamma<0$, for which we provide the proof in detail, and continue to the case $\gamma>0$, for which we give the modifications required due to the presence of the simple pole at $i\gamma$ along the positive imaginary $k_2$-axis.

\vskip 2mm
\noindent
\textbf{Proof of Theorem \ref{2d-rnls-pure-t} for $\gamma < 0$.}
Parametrizing the contour $\mathcal C = \p D$ (see Figure \ref{2d-rnls-dplus-f}), we write   $v = v_1 + v_2$ with 
\begin{align}
&\!\!v_1(x_1, x_2, t)
=
-\frac{1}{(2\pi)^2}
\int_{k_1\in\mathbb R}
\int_{k_2=0}^\infty
e^{ik_1x_1-k_2x_2-i(k_1^2-k_2^2)t}
\, \frac{2k_2}{k_2-\gamma} \, \widehat g(k_1, -k_1^2+k_2^2) \, dk_2 dk_1,
\label{2d-v1-def}
\\
&\!\!v_2(x_1, x_2, t)
=
-\frac{i}{(2\pi)^2}
\int_{k_1\in\mathbb R}\int_{k_2=0}^\infty
e^{ik_1 x_1+ik_2x_2-i(k_1^2+k_2^2)t}
 \frac{2k_2}{k_2-i\gamma} \, \widehat g( k_1, -k_1^2-k_2^2) \, dk_2 dk_1,
\label{2d-v2-def}
\end{align}
and estimate $v_1$ and $v_2$  individually. 
We begin with $v_2$, which involves a purely oscillatory exponential and hence can be handled as a globally defined function via Plancherel's theorem for the Fourier transform. Then, we proceed to $v_1$, which  does not make sense for $x_2<0$ and hence requires a different treatment via the $L^2$ boundedness of the Laplace transform.
\vskip 3mm
\noindent
\textit{Estimation of $v_2$.}
Since $v_2$ makes sense for all $(x_1, x_2) \in \mathbb R^2$, by the definition of the $H^s(\mathbb R^2)$ norm and the fact that $(1+k_1^2+k_2^2)^s \lesssim \big(1+k_1^2\big)^s + (k_2^2)^s$ for any $s\in\mathbb R$, we have
\begin{align}
\left\|v_2(t)\right\|_{H^s(\mathbb R_{x_1}\times\mathbb R_{x_2}^+)}^2 
&\leqslant 
\int_{k_1\in\mathbb R}
\int_{k_2 = 0}^\infty \big(1+k_1^2+k_2^2\big)^s \dfrac{4k_2^2}{\left|k_2-i\gamma\right|^2} \, 
\big| \widehat g(k_1, -k_1^2-k_2^2)\big|^2 dk_2 dk_1
\nn\\
&\lesssim
\int_{k_1\in\mathbb R}
\int_{k_2 = 0}^\infty 
\Big[
\big(1+k_1^2\big)^s +(k_2^2)^s
\Big]  
\frac{k_2^2}{k_2^2+\gamma^2} \, 
\big| \widehat g(k_1, -k_1^2-k_2^2)\big|^2 dk_2 dk_1.
\nn
\end{align}
Furthermore, making the change of variable $k_2 = \big(-\tau-k_1^2\big)^{\frac 12}$ and breaking the resulting $\tau$ integral near and away from $-k_1^2$,  we obtain
\begin{subequations}\label{2d-rnls-pure-temp2}
\begin{align}
\left\|v_2(t)\right\|_{H^s(\mathbb R_{x_1}\times\mathbb R_{x_2}^+)}^2 
&\lesssim
\int_{k_1\in\mathbb R}
\int_{\tau = -\infty}^{-1-k_1^2} 
\Big[
\big(1+k_1^2\big)^s + \big|\tau+k_1^2\big|^s 
\Big]
\frac{\big|\tau+k_1^2\big|^{\frac 12} \, \big| \widehat g(k_1,\tau)\big|^2}{\big|\tau+k_1^2\big| + \gamma^2} \, 
 d\tau dk_1
\label{nls-hl-v2-split-1}
\\
&\hspace*{-0.5cm}
+\int_{k_1\in\mathbb R}
\int_{\tau = -1-k_1^2}^{-k_1^2} 
\Big[
\big(1+k_1^2\big)^s + \big|\tau+k_1^2\big|^s 
\Big]
   \frac{\big|\tau+k_1^2\big|^{\frac 12} \, \big| \widehat g(k_1,\tau)\big|^2}{\big|\tau+k_1^2\big| + \gamma^2} \, 
\big|  d\tau dk_1.
\label{nls-hl-v2-split-0}
\end{align}
\end{subequations}
For $\big|\tau+k_1^2\big| \geqslant 1$,  we have 
$
\frac{\big|\tau+k_1^2\big|^{\frac 12}}{\big|\tau+k_1^2\big| + \gamma^2}
\lesssim 
\big(1+\big|\tau+k_1^2\big|\big)^{-\frac 12}
$
and
$
\big|\tau+k_1^2\big|^s \lesssim \big(1+\big|\tau+k_1^2\big|\big)^s
$, $s\in \mathbb R$.
Thus, noting also that
$1+\big|\tau+k_1^2\big| \simeq \big(1+\big|\tau+k_1^2\big|^2\big)^{\frac 12}$,
the first of the above integrals becomes
\begin{align}\label{nls-hl-v2-far}
\eqref{nls-hl-v2-split-1}
&\lesssim
\int_{k_1\in\mathbb R} \big(1+k_1^2\big)^s \int_{\tau=-\infty}^{-1-k_1^2}  \left(1+\big|\tau+k_1^2\big|^2\right)^{-\frac 14}
 \big| \widehat g(k_1, \tau)\big|^2 d\tau dk_1
\nn\\
&\quad
+\int_{k_1\in\mathbb R}  \int_{\tau=-\infty}^{-1-k_1^2}  \left(1+\big|\tau+k_1^2\big|^2\right)^{\frac{2s-1}{4}}
 \big| \widehat g(k_1, \tau)\big|^2 d\tau dk_1, \quad s \in \mathbb R. 
\end{align}
For  $\big|\tau+k_1^2\big|\leqslant 1$ and $\gamma \neq 0$, we have  
\begin{equation}\label{1/gamma}
\big(|\tau+k_1^2| + \gamma^2\big)^{-1}
 \leqslant 
\max\big\{\tfrac{1}{\gamma^2}, 1\big\} 
\big(1 + |\tau+k_1^2| \big)^{-1}
\end{equation}
(this is not true for $\gamma=0$, which is why the  Neumann problem is treated separately in Section~\ref{neumann-s}) and
$
\big|\tau+k_1^2\big|^{s+\frac 12} \leqslant \big(1+\big|\tau+k_1^2\big|\big)^{s+\frac 12}
$
for $s\geqslant -\tfrac 12$.
Thus,  since 
$
1+\big|\tau+k_1^2\big| \simeq \big(1+\big|\tau+k_1^2\big|^2\big)^{\frac 12}
$, 
for $s\geqslant -\frac 12$ we obtain
\begin{align}\label{nls-hl-v2-near}
\eqref{nls-hl-v2-split-0}
&\lesssim
\int_{k_1\in\mathbb R} \big(1+k_1^2\big)^s
\int_{\tau = -1-k_1^2}^{-k_1^2}   
\left(1+\big|\tau+k_1^2\big|^2\right)^{-\frac 14}
\,
 \big| \widehat g(k_1, \tau)\big|^2  d\tau dk_1
\nn\\
&\quad
+
\int_{k_1\in\mathbb R}
\int_{\tau = -1-k_1^2}^{-k_1^2}\left(1+\big|\tau+k_1^2\big|^2\right)^{\frac{2s-1}{4}}
\,
 \big| \widehat g(k_1, \tau)\big|^2     d\tau dk_1.
\end{align}
Combining estimates \eqref{nls-hl-v2-far} and \eqref{nls-hl-v2-near},  for  $s\geqslant -\frac 12$, $\gamma<0$ and $t\in [0, 2]$  we deduce
\begin{align}\label{2d-nls-v2-se-r}
\left\| v_2(t) \right\|_{H^s(\mathbb R_{x_1} \times \mathbb R_{x_2}^+)}^2
&\lesssim
\int_{k_1\in\mathbb R} \big(1+k_1^2\big)^s
\int_{\tau = -\infty}^{-k_1^2}   
\left(1+\big|\tau+k_1^2\big|^2\right)^{-\frac 14}
\,
 \big| \widehat g(k_1, \tau)\big|^2  d\tau dk_1
\nn\\
&
+
\int_{k_1\in\mathbb R}
\int_{\tau = -\infty}^{-k_1^2}\left(1+\big|\tau+k_1^2\big|^2\right)^{\frac s2-\frac 14}
\,
 \big| \widehat g(k_1, \tau)\big|^2     d\tau dk_1.
\end{align}

\noindent
\textit{Estimation of $v_1$.} 
As the expression \eqref{2d-v1-def} for $v_1$ only makes sense for $x_2>0$, we will estimate it by employing the definition of the Sobolev norm in terms of derivatives in $L^2$. In particular, restricting $s\geqslant 0$ we have
\begin{equation}\label{2d-nls-hs-l2-frac}
\left\|v_1(t)\right\|_{H^s(\mathbb R_{x_1}\times \mathbb R_{x_2}^+)}^2
=
\sum_{|\mu|\leqslant \left\lfloor s\right\rfloor}
\left\|\p_x^\mu v_1(t)\right\|_{L^2(\mathbb R_{x_1}\times \mathbb R_{x_2}^+)}^2
+
\sum_{|\mu|=\left\lfloor s\right\rfloor}
\left\|\p_x^{\mu} v_1(t)\right\|_\beta^2 
\end{equation}
where for $x=(x_1, x_2)$, $\mu = (\mu_1, \mu_2)$ we denote $\p_x^\mu = \p_{x_1}^{\mu_1}\p_{x_2}^{\mu_2}$, $|\mu| = \mu_1+\mu_2$, and  for $\beta=s-\left\lfloor s \right\rfloor\in (0, 1)$ we define the fractional norm
\begin{equation}\label{2d-frac-norm-def}
\left\|v_1(t)\right\|_\beta^2
=
\int_{x\in\mathbb R \times \mathbb R^+} 
\int_{y\in\mathbb R \times \mathbb R^+} 
\frac{\left|v_1(x, t)  - v_1(y, t) \right|^2}{\left|x-y\right|^{2(1+\beta)}}
\, dy dx.
\end{equation}

We begin with the integer part of the norm \eqref{2d-nls-hs-l2-frac} and, more specifically, with  $\left\|\p_x^\mu v_1(t)\right\|_{L^2(\mathbb R_{x_1}\times \mathbb R_{x_2}^+)}$ for   $|\mu| = \mu_1+\mu_2 \in\mathbb N \cup \{0\}$ and $|\mu|\leqslant \left\lfloor s \right\rfloor$. 
Differentiating the unified transform expression \eqref{2d-v1-def} for $v_1$, we find
\begin{equation}
\p_x^\mu v_1(x_1, x_2, t) 
\simeq
\int_{k_1\in\mathbb R} \int_{k_2=0}^\infty
e^{ik_1x_1-k_2x_2-i(k_1^2-k_2^2)t}
\, k_1^{\mu_1}  k_2^{\mu_2} \, \frac{k_2}{k_2-\gamma} \, 
\, \widehat g(k_1, -k_1^2 + k_2^2) dk_2 dk_1.
\nn
\end{equation}
Hence, Plancherel's theorem between the $x_1$ and $k_1$ integrals implies
$$
\left\|\p_x^\mu v_1(t)\right\|_{L^2(\mathbb R_{x_1}\times \mathbb R_{x_2}^+)}^2
\simeq
 \int_{k_1\in\mathbb R}
(k_1^2)^{\mu_1}
\bigg\|
\int_{k_2=0}^\infty \! e^{-k_2x_2 + ik_2^2 t} \, 
 \frac{ k_2^{\mu_2+1}}{k_2-\gamma} \, \widehat g(k_1, -k_1^2 + k_2^2) dk_2
\bigg\|_{L^2(\mathbb R_{x_2}^+)}^2
\!\!\!\! dk_1.
$$
Identifying the $k_2$ integral as the Laplace transform of the function
$
e^{ik_2^2 t} \, k_2^{\mu_2} \, \frac{k_2}{k_2-\gamma} \, \widehat g(k_1, -k_1^2 + k_2^2),
$
we estimate the $L^2$ norm of that integral by using the fact (see  \cite{h1929} and Lemma 3.2 in \cite{fhm2017}) that the Laplace transform 
$
\mathcal L: \varphi(k)  \mapsto  \displaystyle \int_{k=0}^{\infty} e^{-kx} \varphi(k) dk
$
is bounded from $L^2(\mathbb R^+)$ into $L^2(\mathbb R^+)$ with
$
\left\|\mathcal L\left\{\varphi\right\}\right\|_{L^2(\mathbb R^+)}
\leqslant
\sqrt{\pi} \left\|\varphi\right\|_{L^2(\mathbb R^+)}$.
Thus, we obtain 
\begin{equation*}
\left\|\p_x^\mu v_1(t)\right\|_{L^2(\mathbb R_{x_1}\times \mathbb R_{x_2}^+)}^2
\lesssim
 \int_{k_1\in\mathbb R}
(k_1^2)^{\mu_1}
\left\|
e^{ik_2^2 t} \, k_2^{\mu_2} \, \frac{k_2}{k_2-\gamma} \, \widehat g(k_1, -k_1^2 + k_2^2) 
\right\|_{L^2(\mathbb R_{k_2}^+)}^2
dk_1
\end{equation*}
and, since
$
\left(k_2-\gamma\right)^2 \geqslant k_2^2 + \gamma^2
$
for 
$k_2\geqslant 0$ and $\gamma\leqslant 0$ 
(note that this is not true when $\gamma>0$),  
\begin{align}
\left\|\p_x^\mu v_1(t)\right\|_{L^2(\mathbb R_{x_1}\times \mathbb R_{x_2}^+)}^2
&\lesssim
\int_{k_1\in\mathbb R} 
\int_{k_2=0}^\infty
(k_1^2)^{\mu_1}(k_2^2)^{\mu_2}
\frac{k_2^2}{k_2^2+\gamma^2} 
\,
\big|
\widehat g(k_1, -k_1^2 + k_2^2)
\big|^2
dk_2 dk_1.
\label{comp+}
\end{align}
Inserting this estimate in the integer part of the Sobolev norm \eqref{2d-nls-hs-l2-frac} and recalling that $|\mu| = \mu_1+\mu_2 \in\mathbb N \cup \{0\}$ and $|\mu|\leqslant \left\lfloor s \right\rfloor$, we find
\begin{equation}\label{2d-nls-v1-temp-r}
\hspace*{-3mm}
\sum_{|\mu|\leqslant \left\lfloor s \right \rfloor} 
\left\|\p_x^\mu v_1(t)\right\|_{L^2(\mathbb R_{x_1}\times \mathbb R_{x_2}^+)}^2
\lesssim
\int_{k_1\in\mathbb R} 
\int_{k_2=0}^\infty
\!\!\!\!
\frac{\big(1+k_1^2+k_2^2\big)^{\left\lfloor s \right\rfloor}
k_2^2 \, \big|
\widehat g(k_1, -k_1^2 + k_2^2)
\big|^2}{k_2^2+\gamma^2} 
\, 
dk_2 dk_1,
\end{equation}
where we have used the binomial theorem twice in order to compute the two sums over $\mu_2$ and $|\mu|$. 
Noting that $(1 + k_1^2 + k_2^2)^{\left\lfloor s \right\rfloor} \simeq  (1 + k_1^2)^{\left\lfloor s \right\rfloor} + (k_2^2)^{\left\lfloor s \right\rfloor}$ and making the change of variable $k_2=\big(\tau+k_1^2\big)^{\frac 12}$, we  handle the right-hand side of \eqref{2d-nls-v1-temp-r}  similarly to \eqref{2d-rnls-pure-temp2} for $v_2$, i.e. by  splitting the range of the $\tau$ integral near and away from $-k_1^2$.  Eventually, this yields
\begin{align}\label{2d-nls-v1-int-est-r}
\sum_{|\mu|\leqslant \left\lfloor s \right \rfloor} 
\left\|\p_x^\mu v_1(t)\right\|_{L^2(\mathbb R_{x_1}\times \mathbb R_{x_2}^+)}^2
&\lesssim
\int_{k_1\in\mathbb R}
\int_{\tau =-k_1^2}^\infty
 \big(1+k_1^2\big)^{\left\lfloor s \right\rfloor}  
\left(1+\big|\tau+k_1^2\big|^2\right)^{-\frac 14}
\,
 \big| \widehat g(k_1, \tau)\big|^2  d\tau dk_1
\nn\\
&\quad
+
\int_{k_1\in\mathbb R}
\int_{\tau =-k_1^2}^\infty
\left(1 + \big|\tau+k_1^2\big|^2\right)^{\frac{2\left\lfloor s \right\rfloor-1}{4}} 
 \big| \widehat g(k_1, \tau)\big|^2     d\tau dk_1
 \end{align}
for all $s\geqslant 0$ and $\gamma<0$ (as for $v_2$, the Neumann case $\gamma = 0$ is treated separately in Section \ref{neumann-s}).

Having completed the estimation of the integer part of the Sobolev norm \eqref{2d-nls-hs-l2-frac}, we turn our attention to the fractional norms $\left\|\p_x^\mu v_1(t)\right\|_\beta$ with $\mu_1+\mu_2=\left\lfloor s \right\rfloor\in\mathbb N \cup \{0\}$. Note that   \eqref{2d-frac-norm-def}  can be expressed in the convenient form
$
\left\|v_1(t)\right\|_\beta^2
\simeq
\int_{x\in\mathbb R \times \mathbb R^+}\int_{z\in\mathbb R \times \mathbb R^+} \frac{\left|v_1(x+z, t)- v_1(x, t) \right|^2}{\left|z\right|^{2(1+\beta)}}\, dz dx
$.
Then, differentiating \eqref{2d-v1-def} and employing Plancherel's theorem for the integrals with respect to $x_1$ and $k_1$ as well as (once again) the Laplace transform boundedness in $L^2(\mathbb R^+)$ for the  integrals with respect to  $x_2$ and $k_2$, we obtain
\begin{equation}\label{2d-i-integral-r}
\left\|\p_x^\mu v_1(t)\right\|_\beta^2
\lesssim
\int_{k_1\in\mathbb R} (k_1^2)^{\mu_1}
\int_{k_2=0}^\infty
\Big| k_2^{\mu_2} \frac{k_2}{k_2-\gamma} \, \widehat g(k_1, -k_1^2 + k_2^2)\Big|^2
I(k_1, k_2, \beta)
dk_2 dk_1,
\end{equation}
where  
$I(k_1, k_2, \beta)
=
\int_{z_1\in \mathbb R}\int_{z_2=0}^\infty 
\frac{\left|e^{ik_1 z_1 - k_2 z_2}-1\right|^2}
{\left(z_1^2+z_2^2\right)^{1+\beta}}
\,
dz_2dz_1$.
By Lemma 2.2 of \cite{hm2020}, for $\beta\in (0, 1)$ we have 
$I(k_1, k_2, \beta)
\lesssim
\big(k_1^2+k_2^2\big)^\beta$.
Hence, \eqref{2d-i-integral-r} becomes
\begin{equation} 
\left\|\p_x^\mu v_1(t)\right\|_\beta^2
\lesssim
\int_{k_1\in\mathbb R} (k_1^2)^{\mu_1}
\int_{k_2=0}^\infty  \big(k_1^2+k_2^2\big)^\beta 
\, \Big| k_2^{\mu_2} \frac{k_2}{k_2-\gamma} \, \widehat g(k_1, -k_1^2 + k_2^2) \Big|^2 dk_2  dk_1
\nn
\end{equation}
and, using   the  inequality $
\left(k_2-\gamma\right)^2 \geqslant k_2^2 + \gamma^2
$  together with the binomial theorem, we find
\begin{equation}
\sum_{|\mu|=\left\lfloor s \right \rfloor}
\left\|\p_x^\mu v_1(t)\right\|_\beta^2
\lesssim
\int_{k_1\in\mathbb R} 
\int_{k_2=0}^\infty  \big(k_1^2+k_2^2\big)^s
\frac{k_2^2}{k_2^2+\gamma^2}
\, \big|\widehat g(k_1, -k_1^2 + k_2^2) \big|^2 dk_2  dk_1.
\nn
\end{equation}
The term on the right-hand side  can be handled like the one in \eqref{2d-nls-v1-temp-r} to yield
\begin{align}\label{2d-nls-v1-frac-est-r}
\sum_{|\mu|=\left\lfloor s \right \rfloor}
\left\|\p_x^\mu v_1(t)\right\|_\beta^2
&\lesssim
\int_{k_1\in\mathbb R} 
\int_{\tau=-k_1^2}^\infty
\big(1+k_1^2\big)^s
\left(1+\big|\tau+k_1^2\big|^2\right)^{-\frac 14}
\left|
\widehat g(k_1, \tau) 
\right|^2
d\tau dk_1  
\\
&\quad
+
\int_{k_1\in\mathbb R} 
\int_{\tau=-k_1^2}^\infty
\left(1+\big|\tau+k_1^2\big|^2\right)^{\frac{2s-1}{4}}
\left|
\widehat g(k_1, \tau)
\right|^2
d\tau dk_1,  \quad s\geqslant 0, \ \gamma<0.
\nn
\end{align}
Overall, estimates \eqref{2d-nls-v1-int-est-r} and \eqref{2d-nls-v1-frac-est-r} combined with the Sobolev norm definition \eqref{2d-nls-hs-l2-frac}  imply
\begin{align}\label{2d-nls-v1-se-r}
\left\| v_1(t) \right\|_{H^s(\mathbb R_{x_1}\times \mathbb R_{x_2}^+)}^2
&\lesssim
\int_{k_1\in\mathbb R} 
\int_{\tau=-k_1^2}^\infty
\big(1+k_1^2\big)^s
\left(1+\big|\tau+k_1^2\big|^2\right)^{-\frac 14}
\left|
\widehat g(k_1, \tau)
\right|^2
d\tau dk_1  
\nn\\
&\quad
+
\int_{k_1\in\mathbb R} 
\int_{\tau=-k_1^2}^\infty
\left(1+\big|\tau+k_1^2\big|^2\right)^{\frac{2s-1}{4}}
\left|
\widehat g(k_1, \tau)
\right|^2
d\tau dk_1
\end{align}
for all $s\geqslant 0$, $\gamma<0$ and $t\in [0, 2]$, which together with estimate \eqref{2d-nls-v2-se-r} for $v_2$ 
yields the desired Hadamard estimate \eqref{2d-rnls-pure-se} in the case $\gamma<0$.

\vskip 2mm
\noindent
\textbf{Proof of Theorem \ref{2d-rnls-pure-t} for $\gamma >0$.} 
We now provide the modifications necessary for proving estimate~\eqref{2d-rnls-pure-se} when $\gamma>0$. 
Recall that the difference between the cases $\gamma<0$ and $\gamma>0$ is that, in the latter case, the complex contour of integration $\mathcal C$ in the unified transform formula \eqref{2d-nls-pure-sol} is given by $\p \widetilde D$ instead of $\p D$ (see Figure~\ref{2d-rnls-dplus-f}), so that  the singularity at $k_2 = i\gamma$ (which for $\gamma>0$ lies along   $\p D$) is avoided by means of $C_{\gamma/2}(i\gamma)$, which denotes the right half of the negatively oriented circle of radius $\gamma/2$ and center at $i\gamma$.
Hence, for $\gamma>0$, the solution of the reduced pure linear ibvp~\eqref{2d-rnls-pure} consists of three parts,
$v  = v_{1, 1} + v_{1, 2} + v_2$, where
\begin{align}
&v_{1, 1}(x_1, x_2, t)
\simeq
\int_{k_1\in\mathbb R}
\left(\int_{k_2=0}^{\frac \gamma 2} + \int_{k_2=\frac {3\gamma}2}^\infty\right)
e^{ik_1x_1-k_2x_2-i(k_1^2-k_2^2)t}
 \frac{k_2}{k_2-\gamma} \, \widehat g(k_1, -k_1^2+k_2^2)  dk_2 dk_1,
\label{2d-v11-def}
\\
&v_{1, 2}(x_1, x_2, t)
\simeq
\int_{k_1\in\mathbb R}\int_{k_2\in C_{\gamma/2}(i\gamma)}
 e^{ik_1 x_1+ik_2x_2-i(k_1^2+k_2^2)t}
 \frac{k_2}{k_2-i\gamma} \, \widehat g( k_1, -k_1^2-k_2^2) \, dk_2 dk_1,
\label{2d-v12-def}
\\
&v_2(x_1, x_2, t)
\simeq
\int_{k_1\in\mathbb R}\int_{k_2=0}^\infty
e^{ik_1 x_1+ik_2x_2-i(k_1^2+k_2^2)t}
\, \frac{k_2}{k_2-i\gamma} \, \widehat g( k_1, -k_1^2-k_2^2) \, dk_2 dk_1.
\label{2d-v22-def}
\end{align}

The term \eqref{2d-v22-def} can be estimated exactly like the term \eqref{2d-v2-def} from the case $\gamma<0$. Also, the first integral in  \eqref{2d-v11-def} is similar to the term \eqref{2d-v1-def}. Indeed,  since $k_2\in [0, \gamma/2]$, we have  
$
(k_2-\gamma)^2 
\geqslant
\frac 18 \big(k_2^2+\gamma^2\big).
$
This inequality plays the role of inequality 
$
\left(k_2-\gamma\right)^2 \geqslant k_2^2 + \gamma^2
$
that was valid for $k_2\geqslant 0$ and $\gamma\leqslant 0$, and hence allows us to handle the first integral in  \eqref{2d-v11-def} just like $v_1$ from the case $\gamma<0$,  eventually obtaining  estimate \eqref{comp+} and, in turn, estimate~\eqref{2d-nls-v1-int-est-r}.
For the second integral in~\eqref{2d-v11-def}, since $|k_2-i\gamma|\geqslant \gamma/2$ we have  
$
\frac{k_2^2}{(k_2-\gamma)^2}
\leqslant 
2 + \frac{2\gamma^2}{(k_2-\gamma)^2}
\leqslant 10$. 
We remark that this crude bound can be used for the integral over $k_2\in[3\gamma/2, \infty)$ but not for the one over  $k_2\in[0, \gamma/2]$. This is because, in the former range, we can handle the term $|\tau+k_1^2|^{-\frac 12}$  arising from the change of variable $k_2 = (-\tau-k_1^2)$ via the inequality
$|\tau+k_1^2|^{-\frac 12}
\lesssim
\max\big\{\tfrac{1}{\gamma}, 1\big\} \big(1+\big|\tau+k_1^2\big|^2\big)^{-\frac 14}$,
which follows from the fact that $k_2 \geqslant  3\gamma/2$. 
Overall, we conclude that the term  \eqref{2d-v11-def} admits the same estimate with the term $v_1$ appearing in the case $\gamma<0$.

The remaining component $v_{1,2}$ of the reduced pure ibvp solution, which is defined by \eqref{2d-v12-def}, has not appeared before, since it involves for the first time the semicircular contour $C_{\gamma/2}(i\gamma)$. This is a finite contour that stays a fixed distance $\gamma/2$ away from the singularity at $k_2=i\gamma$ and so, in principle, the estimation of $v_{1,2}$ should go through without any issues. However, it turns out that some technical details are needed. The main reason for this is that the changes of variable $\tau = -k_1^2 \pm k_2^2$, which have been used in all of the previous estimations, would now result in complex values for $\tau$, thus making it difficult to relate the relevant  integrals to the Bourgain norms~\eqref{xsb-def}. Thus, instead of making these changes of variables, we will exploit the boundedness of the contour $C_{\gamma/2}(i\gamma)$ together with the compact support in $t$ of the reduced pure linear ibvp datum $g(x_1, t)$ in order to estimate $v_{1,2}$ in a different way.
More specifically, parametrizing $C_{\gamma/2}(i\gamma)$ by $k_2=k_2(\theta) = i\gamma + \frac \gamma 2 e^{i\theta}$, we have
\begin{equation*}
v_{1, 2}(x_1, x_2, t)
\simeq
\int_{k_1\in\mathbb R}\int_{\theta=-\frac \pi2}^{\frac \pi2}
e^{ik_1 x_1+ik_2(\theta)x_2-i(k_1^2+k_2(\theta)^2)t}
\,  k_2(\theta) \, \widehat g(k_1, -k_1^2-k_2(\theta)^2) \,  d\theta dk_1.
\end{equation*}
We will estimate this integral by using the norm \eqref{2d-nls-hs-l2-frac}. 
For the integer part of that norm, we have $\p_x^\mu = \p_{x_1}^{\mu_1}\p_{x_2}^{\mu_2}$ with $|\mu|=\mu_1+\mu_2\in \mathbb N\cup \{0\}$ and $|\mu|\leqslant \left\lfloor s \right\rfloor$. Thus,  differentiating the above expression for $v_{1, 2}$ and then using Plancherel's theorem in $x_1$ and $k_1$, Minkowski's integral inequality for the $x_2$ and $\theta$ integrals, and the fact that
$
\int_{x_2=0}^\infty \big|e^{ik_2(\theta)x_2}\big|^2 dx_2
=
\left[\gamma \left(2+\sin \theta\right)\right]^{-1} = c_\gamma < \infty,
$
we can bound $\left\|\p_x^\mu v_{1, 2}(t)\right\|_{L^2(\mathbb R_{x_1}\times \mathbb R_{x_2}^+)}^2$ by
\begin{align}\label{conv-w-0}
\int_{k_1\in\mathbb R}
(k_1^2)^{\mu_1}
\bigg[
\int_{\theta=-\frac \pi2}^{\frac \pi2}
\left|
e^{-ik_2(\theta)^2t}
k_2(\theta)^{\mu_2+1} \, \widehat g(k_1, -k_1^2-k_2(\theta)^2)
\right| \left(2+\sin\theta\right)^{-\frac 12} d\theta \bigg]^2   dk_1.
\end{align}
Since $0\leqslant \mu_1\leqslant \lfloor s \rfloor$, the right-hand side of the above inequality will have to be controlled by the norm of the Bourgain space $X^{s, -\frac 14}$ (as opposed to $X^{0, \frac{2s-1}{4}}$). In order to accomplish this, we will introduce a multiplier $\simeq 1$ in the above $\theta$ integral. More specifically, let  
$$
\psi(t) = \left\{\begin{array}{ll} e^{-t}, &t\in (0, 2) \\ 0, &t\in [0, 2]^c \end{array}\right\}
\ \Rightarrow \
\widehat \psi(\tau) = \frac{1-e^{-2(1+i\tau)}}{1+i\tau}
$$
and observe that, for $k_2 \in C_{\gamma/2}(i\gamma)$, we have
$\big|\widehat\psi(-k_2^2)\big|
=
\big|
\frac{1-e^{2(ik_2^2-1)}}{1-ik_2^2}
\big|^2
\simeq c_\gamma \simeq 1
$.
Indeed, since
$\text{Im}(k_2^2)\geqslant 0$ along $C_{\gamma/2}(i\gamma)$,  it follows that $\big|1-e^{2(ik_2^2-1)}\big|\simeq 1$. 
Moreover, $\big|1-ik_2^2\big| = c_\gamma \simeq 1$
since 
$\big(1+\frac{\gamma^4}{16}\big)^{\frac 12} \leqslant \big|1-ik_2^2\big| \leqslant  1+ \frac{9\gamma^2}{4}$.
Therefore,
\begin{align}\label{px-nst}
\eqref{conv-w-0}
&\lesssim
\int_{k_1\in\mathbb R}
(k_1^2)^{\mu_1}
\!
\bigg[
\int_{\theta=-\frac \pi2}^{\frac \pi2}
\left|
e^{-ik_2(\theta)^2t}
k_2(\theta)^{\mu_2+1} \, \widehat\psi(-k_2(\theta)^2) \cdot  \widehat g(k_1, -k_1^2-k_2(\theta)^2)
\right| 
\left(2+\sin\theta\right)^{-\frac 12} d\theta \bigg]^2 \!  \!dk_1
\nn\\
&\lesssim
\int_{k_1\in\mathbb R}
(k_1^2)^{\mu_1}
\sup_{\theta\in[-\frac \pi 2, \frac \pi 2]}
\left|
\widehat\psi(-k_2(\theta)^2) \cdot  \widehat g(k_1, -k_1^2-k_2(\theta)^2)
\right|^2 dk_1,
\end{align}
since the $\theta$ integral remaining after taking the above supremum is bounded by some $c_{s, \gamma} < \infty$. 

The supremum in \eqref{px-nst} can be bounded after noting that, by the convolution property of the Fourier transform and the compact support in $t$ of both $\psi$ and $g$,  
$$
\widehat \psi(-k_2^2)\cdot \widehat g(k_1, -k_1^2-k_2^2)
=
\int_{t=0}^4
e^{i(k_1^2+k_2^2)t} \big\{ e^{-ik_1^2 \cdot} \psi(\cdot) * \widehat g^{x_1}(k_1, \cdot)\big\} (t) dt.
$$
Hence, by  the Cauchy-Schwarz inequality, the fact that $\text{Im}(k_2^2)\geqslant 0$ along $C_{\gamma/2}(i\gamma)$, and Plancherel's theorem, we obtain
$$
\left|
\widehat \psi(-k_2^2)\cdot \widehat g(k_1, -k_1^2-k_2^2)
\right|
\lesssim
\left\| \mathcal F_t \big\{e^{-ik_1^2 t} \psi(t)\big\}(\tau) \cdot \widehat g(k_1, \tau)\right\|_{L^2(\mathbb R_\tau)}.
$$
Then, substituting
$
\mathcal F_t \big\{e^{-ik_1^2 t} \psi(t)\big\}(\tau)
=
\frac{1-e^{-2(1+i(\tau+k_1^2))}}{1+i(\tau+k_1^2)}
$
and noting that for $\tau, k_1 \in \mathbb R$ we have
$
\big|1-e^{-2(1+i(\tau+k_1^2))}\big|
\leqslant 1 + e^{-2}
$,
we find
\begin{equation*}
\left|
\widehat \psi(-k_2^2)\cdot \widehat g(k_1, -k_1^2-k_2^2)
\right|^2
\lesssim
\int_{\tau\in\mathbb R}
\left(1+\big(\tau+k_1^2\big)^2\right)^{-1} \left|\widehat g(k_1, \tau)\right|^2 d\tau.
\end{equation*}
Inserting this estimate in  \eqref{px-nst}, we deduce
\begin{align}\label{frac-pos-int-1}
\left\|\p_x^\mu v_{1, 2}(t)\right\|_{L^2(\mathbb R_{x_1}\times \mathbb R_{x_2}^+)}^2
&\lesssim
\int_{k_1\in\mathbb R}
(k_1^2)^{\mu_1}
\int_{\tau\in\mathbb R}
\left(1+\big(\tau+k_1^2\big)^2\right)^{-1} \left|\widehat g(k_1, \tau)\right|^2 d\tau dk_1
\nn\\
&\simeq
\left\|g\right\|_{X^{\mu_1, -1}}^2
\leqslant
\left\|g\right\|_{X^{\left\lfloor s \right\rfloor, -1}}^2
\leqslant
\left\|g\right\|_{X^{s, -\frac 14}}^2
\end{align}
for all $t\in (0, 2)$ and $|\mu|=\mu_1 + \mu_2 \in \mathbb N\cup\{0\}$ with $|\mu| \leqslant \left\lfloor s \right\rfloor$, 
which gives the desired estimate for the integer part of the norm \eqref{2d-nls-hs-l2-frac}.

The fractional  norm $\left\|\p_x^\mu v_{1, 2}(t)\right\|_\beta$   also satisfies \eqref{frac-pos-int-1} for each $|\mu| = \mu_1 + \mu_2 = \left\lfloor s \right \rfloor \in \mathbb N \cup \{0\}$. This can be shown via the same steps with the integer part  above together with the bound
$$
\int_{z_1\in\mathbb R} \int_{z_2=0}^\infty
\frac{\big|e^{ik_1 z_1+ik_2 z_2}-1\big|^2}{\left(z_1^2+z_2^2\right)^{1+\beta}}\, 
dz_2  dz_1 \lesssim \max\big\{1, (k_1^2)^\beta\big\},
$$
which holds for all $k_1\in\mathbb R$, $k_2 \in C_{\gamma/2}(i\gamma)$ and $\beta \in (0, 1)$.
This bound  can be proved similarly to Lemma~2.2 in \cite{hm2020} by considering the cases $|k_1|\geqslant 1$ and $|k_1|\leqslant 1$  separately and then employing the inequality 
$\left|e^z-1\right| \leqslant  |z| e^{|z|}$, $z\in\mathbb C$. 
This completes the proof of Theorem \ref{2d-rnls-pure-t} for $\gamma>0$.~\hfill $\square$

\section{Estimates for the linear Schr\"odinger equation on the plane}
\label{ivp-s}

In this section, we establish various estimates for the linear Schr\"odinger initial value problem (ivp) on the plane. These results will be combined in Section \ref{fls-ibvp-s} with those of Section \ref{2d-rnls-pure-s} on the reduced pure linear ibvp in order to prove Theorem \ref{2d-rfls-t} for the forced linear ibvp \eqref{2d-rfls-ibvp}, which is the basis for proving the well-posedness Theorem \ref{2d-rnls-t} for the NLS ibvp \eqref{2d-rnls-ibvp}. In particular, besides the Hadamard solution space, we will obtain estimates for the linear ivp in the Bourgain-type spaces \eqref{xsbt-def} motivated by Theorem \ref{2d-rnls-pure-t}.
 
We begin with the homogeneous ivp
\begin{equation}\label{2d-rls-ivp}
\begin{aligned}
&i U _t+  U_{x_1x_1} + U_{x_2x_2}  =0,  \quad x = (x_1, x_2) \in \mathbb R^2, \ t \in \mathbb R,
\\
&U (x, 0)= U_0(x)\in H^s(\mathbb R^2),
\end{aligned}
\end{equation}
whose solution is given by
\begin{equation}\label{2d-nls-ls-ivp-sol}
U (x, t) 
= 
S\big[U_0; 0\big] (x, t) 
= 
\frac{1}{(2\pi)^2} 
\int_{k\in\mathbb R^2} e^{ik\cdot x-i k^2 t} \,  \widehat U_0(k) dk,
\end{equation}
where $\widehat U_0$ denotes the Fourier transform of $U_0$ on the whole plane, $\widehat U_0(k) 
= 
\int_{x \in \mathbb R^2} e^{-ik \cdot x}\, U_0(x) dx$,
with $k = (k_1, k_2) \in \mathbb R^2$,  $k \cdot x = k_1x_1+k_2x_2$ and $k^2 = k_1^2 + k_2^2$.
\begin{theorem}[\textcolor{blue}{Homogeneous linear ivp}]
\label{2d-rls-ivp-t}
The solution $U=S\big[U_0; 0\big]$ to the linear Schr\"odinger ivp~\eqref{2d-rls-ivp}, given by the Fourier transform formula \eqref{2d-nls-ls-ivp-sol}, satisfies the   isometry relation
\begin{equation}\label{2d-rls-ivp-se}
\sup_{t\in [0, T]}
\left\|U(t)\right\|_{H^s(\mathbb R_x^2)} 
=
\left\|U_0\right\|_{H^s(\mathbb R_x^2)}, \quad s\in\mathbb R,
\end{equation}
as well as the Bourgain-type estimates
\begin{alignat}{2}
&\sup_{x_2\in\mathbb R}
\left\|\left(U_{x_2} + \gamma U\right)(x_2)\right\|_{X_T^{0, \frac{2s-1}{4}}} 
&\leqslant c_s \left\|U_0\right\|_{H^s(\mathbb R_x^2)},  \quad s\geqslant 1,
\label{2d-rls-te}
\\
&\sup_{x_2\in\mathbb R}
\left\|\left(U_{x_2} + \gamma U\right)(x_2)\right\|_{X_T^{s, -\frac 14}} 
&\leqslant 
c_{s, \gamma} \left\|U_0\right\|_{H^s(\mathbb R_x^2)}, \quad    s\geqslant 0. 
\label{2d-rls-te-2}
\end{alignat}
\end{theorem}

\begin{proof}
The isometry relation  \eqref{2d-rls-ivp-se} follows directly from formula \eqref{2d-nls-ls-ivp-sol}  and the definition of the Sobolev norm.
Estimate \eqref{2d-rls-te} can be deduced by estimate (3.5)  in \cite{hm2020}, according to which
$
\sup_{x_2\in \mathbb R} \left\|U(x_2)\right\|_{X_T^{0, \frac{2s+1}{4}}} 
\leqslant c_s \left\|U_0\right\|_{H^s(\mathbb R_x^2)}
$, $s\geqslant 0$.
In particular, observe that $U_{x_2} = S\big[\p_{x_2} U_0; 0\big]$ so it suffices to employ that estimate with $U_{x_2}$ in place of $U$ and with $s$ replaced by $s-1$.

Concerning estimate \eqref{2d-rls-te-2}, we note that (3.5) in \cite{hm2020} additionally implies the estimate
$
\sup_{x_2 \in \mathbb R} \left\| U(x_2) \right\|_{X_T^{s, \frac 14}}
\leqslant
c_s \left\| U_0 \right\|_{H^s(\mathbb R_x^2)}
$
for $s\geqslant 0$, 
which takes care of the second term in the norm on the left-hand side of~\eqref{2d-rls-te-2}.
Furthermore, denoting by $\dot H^{-\frac 14}$  the homogeneous counterpart of the Sobolev space $H^{-\frac 14}$ and using the fact that $\dot H^{-\frac 14} \subset H^{-\frac 14}$, we have
\begin{equation}\label{2d-rnls-hbour-hom-0}
\left\|U_{x_2}(x_2)\right\|_{X_T^{s, -\frac 14}}^2
\leqslant 
\int_{k_1\in\mathbb R}   
\big(1+k_1^2\big)^s 
\left\| e^{ik_1^2\cdot} \, \widehat U_{x_2}^{x_1}(k_1, x_2, \cdot) \right\|_{\dot H^{-\frac 14}(\mathbb R_t)}^2 dk_1.
\end{equation}
Then,  we use the solution formula \eqref{2d-nls-ls-ivp-sol} together with the change of variable $\tau = -k_2^2$ to write
$$
e^{ik_1^2 t} \, \widehat U_{x_2}^{x_1}(k_1, x_2, t)
=
\frac{1}{4i\pi} \int_{\tau=-\infty}^0 e^{i\tau t}
\left[
e^{-i\sqrt{-\tau}x_2}  \,  \widehat U_0(k_1, -\sqrt{-\tau}) 
-
e^{i\sqrt{-\tau}x_2} \, \widehat U_0(k_1, \sqrt{-\tau})
\right] d\tau.
$$
Hence, by the usual definition of the $\dot H^{-\frac 14}(\mathbb R)$ norm, we find
$$
\left\|e^{ik_1^2 \cdot} \, \widehat U_{x_2}^{x_1}(k_1, x_2, \cdot)\right\|_{\dot H^{-\frac 14}(\mathbb R_t)}^2 
\lesssim
 \int_{\tau=-\infty}^0 |\tau|^{-\frac 12} \left[
 \big|\widehat U_0(k_1, -\sqrt{-\tau}) \big|^2
+
\big|\widehat U_0(k_1, \sqrt{-\tau}) \big|^2 
\right] d\tau
$$
and, changing variable from $\tau$ back to $k_2$, we obtain
\begin{equation}
\left\|e^{ik_1^2 \cdot} \, \widehat U_{x_2}^{x_1}(k_1, x_2, \cdot)\right\|_{\dot H^{-\frac 14}(\mathbb R_t)}^2 
\lesssim
\int_{k_2\in \mathbb R}  \big|\widehat U_0(k_1, k_2) \big|^2 dk_2.
\nn
\end{equation}
In turn,   \eqref{2d-rnls-hbour-hom-0}  yields
$
\left\|U_{x_2}(x_2)\right\|_{X_T^{s, -\frac 14}}^2
\lesssim
\int_{k_1\in\mathbb R} 
\big(1+k_1^2\big)^s 
 \int_{k_2\in \mathbb R}  \big|\widehat U_0(k_1, k_2) \big|^2 dk_2 dk_1
\leqslant
\left\|U_0\right\|_{H^s(\mathbb R_x^2)}^2
$
with the second inequality due to the fact that $s\geqslant 0$.
\end{proof}

We proceed to the forced linear ivp with zero initial data, 
\begin{subequations}\label{2d-rfls-ivp}
\begin{align}
&i W_t+  W_{x_1x_1} + W_{x_2x_2} = F(x, t),\quad  x = (x_1, x_2) \in \mathbb R^2, \ t \in \mathbb R,
\label{2d-nls-flsivp-eq}
\\
& W(x, 0)=0, \quad  x \in \mathbb R^2, \label{2d-nls-flsivp-ic}
\end{align}
\end{subequations}
whose solution is given by
\begin{equation}\label{2d-nls-fls-ivp-sol}
W(x, t) 
= 
S\big[0; F\big](x, t)
=
-\frac{i}{(2\pi)^2}\int_{t'=0}^t   \int_{k \in\mathbb R^2} e^{ik \cdot x - ik^2(t-t')}
\widehat F^x (k, t') dk dt',
\end{equation}
where  
$\widehat F^x(k, t) 
= 
\int_{x \in \mathbb R^2} e^{-ik \cdot x}\, F(x, t) dx
$  is the spatial Fourier transform of $F$ on the whole plane.

\begin{theorem}[\textcolor{blue}{Forced linear ivp with zero initial data}]
\label{2d-rfls-ivp-t}
The solution $W=S\big[0; F\big]$ to the forced linear Schr\"odinger ivp \eqref{2d-rfls-ivp}, given by the Fourier transform formula \eqref{2d-nls-fls-ivp-sol}, satisfies the Hadamard space estimate
\begin{equation}\label{2d-rfls-ivp-se}
\sup_{t\in [0, T]} \left\|W(t)\right\|_{H^s(\mathbb R_x^2)} 
\leqslant
T \sup_{t\in [0, T]} \left\| F(t) \right\|_{H^s(\mathbb R_x^2)},  \quad s\in\mathbb R,
\end{equation}
and the Bourgain-type estimates
\begin{align}
&\sup_{x_2\in\mathbb R}
\left\|\left(W_{x_2} + \gamma W\right)(x_2)\right\|_{X_T^{0, \frac{2s-1}{4}}} 
\leqslant 
c_s  \sqrt T \sup_{t\in [0, T]} \left\|F(t)\right\|_{H^s(\mathbb R_x^2)},  
\quad 1 \leqslant s \leqslant \tfrac 52, \ s\neq \tfrac 32,
\label{2d-rfls-ivp-te}
\\
&\sup_{x_2\in\mathbb R} \left\|\left(W_{x_2} + \gamma W\right)(x_2)\right\|_{X_T^{s, -\frac 14}}  
\leqslant c_{s, \gamma}  \sqrt T \sup_{t\in [0, T]} \left\| F(t) \right\|_{H^s(\mathbb R_x^2)},  \quad    s \geqslant 0.
\label{2d-rfls-ivp-te-2}
\end{align}
\end{theorem}

Before giving the proof of the above results, we remark that estimates \eqref{2d-rls-te}, \eqref{2d-rls-te-2}, \eqref{2d-rfls-ivp-te} and~\eqref{2d-rfls-ivp-te-2} along with estimate \eqref{2d-rnls-pure-se}  for the reduced pure linear ibvp  confirm that the space $B_T^s$ defined by \eqref{bst-def} is the correct space for the boundary data $g$ of the forced linear ibvp \eqref{2d-rfls-ibvp} (and hence of the nonlinear ibvp \eqref{2d-rnls-ibvp}).

\begin{proof}
Estimate \eqref{2d-rfls-ivp-se} is (3.21) from \cite{hm2020}.
Estimate \eqref{2d-rfls-ivp-te} follows from  (3.25)   in \cite{hm2020}, which reads
\begin{equation}\label{hm2020-fls-est-0}
\sup_{x_2\in\mathbb R}
\left\|W(x_2)\right\|_{X_T^{0, \frac{2s+1}{4}}} 
\leqslant 
c_s \sqrt T \sup_{t\in [0, T]}  \left\|F(t)\right\|_{H^s(\mathbb R_x^2)}, \quad 
0 \leqslant s \leqslant \tfrac 32, \ s\neq \tfrac 12,
\end{equation}
after noting that $W_{x_2} = S\big[0; F_{x_2}\big]$ and employing \eqref{hm2020-fls-est-0} for $W$ and $W_{x_2}$ with $s-1$ instead of $s$.

Concerning estimate \eqref{2d-rfls-ivp-te-2}, we first note that   (3.26) in  \cite{hm2020}  yields
\begin{equation}\label{W-hm-est}
\sup_{x_2\in\mathbb R}
\left\|W(x_2)\right\|_{X_T^{s, -\frac 14}}
\leqslant
\sup_{x_2\in\mathbb R}
\left\|W(x_2)\right\|_{X_T^{s, \frac 14}}
\lesssim
\sqrt T \sup_{t\in [0, T]} \left\|F(t)\right\|_{H^s(\mathbb R_x^2)}, \quad s\geqslant 0.
\end{equation}
Moreover, as in the proof of Theorem \ref{2d-rls-ivp-t}, we have
\begin{equation}\label{wx2-temp}
\left\|W_{x_2}(x_2)\right\|_{X_T^{s, -\frac 14}}^2 
\leqslant
\int_{k_1\in \mathbb R} \big(1+k_1^2\big)^s \left\|e^{ik_1^2 t} \, \widehat W_{x_2}^{x_1}(k_1, x_2, t)\right\|_{\dot H^{-\frac 14}(0, T)}^2 dk_1.
\end{equation}
For $k_1, x_2 \in \mathbb R$ and $t\in [0, T]$, formula \eqref{2d-nls-fls-ivp-sol} implies $e^{ik_1^2 t} \, \widehat W_{x_2}^{x_1}(k_1, x_2, t) =  R(k_1, x_2, t)$, where
\begin{equation}
R(k_1, x_2, t)
=
-\frac{i}{2\pi} \int_{t'=0}^t \chi_{[0, T]}(t') \int_{k_2\in\mathbb R} e^{ik_2x_2-ik_2^2(t-t')} \,    \widehat{\p_{x_2} F}^x(k_1, k_2, t') dk_2 dt'.
\end{equation}
The important observation is that, for each $k_1\in \mathbb R$,   $R(k_1, x_2, t)$ satisfies the one-dimensional ivp
\begin{equation}\label{1d-fls-ivp}
\begin{aligned}
&iR_t + R_{x_2x_2} =  \chi_{[0,T]}(t)\, \widehat{\p_{x_2} F}^{x_1}(k_1, x_2, t), \quad x_2 \in\mathbb R, \ t \in \mathbb R,
\\
&R(k_1, x_2, 0) = 0, \quad x_2\in\mathbb R.
\end{aligned}
\end{equation}
In this connection, from the proof of Lemma 11 in \cite{h2005}, we have
\begin{equation}\label{holmer-est}
\sup_{x_2\in\mathbb R} \left\|R(k_1, x_2)\right\|_{\dot H^{-\frac 14}(\mathbb R_t)}
\leqslant
c_s \left\|\chi_{[0,T]}(t)\, \widehat{\p_{x_2} F}^{x_1}(k_1, x_2, t)\right\|_{L^1(\mathbb R_t; \dot H^{-1}(\mathbb R_{x_2}))}.
\end{equation}
Combining this bound with \eqref{wx2-temp} and  Minkowski's integral inequality,  we find
$$
\sup_{x_2\in \mathbb R} 
\left\|W_{x_2}(x_2)\right\|_{X_T^{s, -\frac 14}}
\lesssim
\int_{t=0}^T \left(\int_{k_1\in\mathbb R} \big(1+k_1^2\big)^s
 \int_{k_2\in\mathbb R}  \big|\widehat F^x(k_1, k_2, t)\big|^2 dk_2 dk_1\right)^{\frac 12} dt,
$$
which completes the proof of estimate \eqref{2d-rfls-ivp-te-2} in view of  $\big(1+k_1^2\big)^s \leqslant \big(1+k_1^2+k_2^2\big)^s$ for $s\geqslant 0$.
\end{proof}

%
%
%
%
%
%
%
%
\section{The forced linear ibvp and proof of Theorems \ref{2d-rfls-t} and \ref{2d-rnls-t}}
\label{fls-ibvp-s}

In this section, for $1\leqslant s < \frac 32$ we combine the estimates for the reduced pure linear ibvp  (Theorem~\ref{2d-rnls-pure-t}) and the homogeneous and forced linear ivps (Theorems~\ref{2d-rls-ivp-t} and \ref{2d-rfls-ivp-t}) in order to prove Theorem \ref{2d-rfls-t} for the forced linear ibvp \eqref{2d-rfls-ibvp} and the Hadamard well-posedness Theorem~\ref{2d-rnls-t} for the nonlinear ibvp \eqref{2d-rnls-ibvp}. For this purpose, we decompose ibvp~\eqref{2d-rfls-ibvp} into component problems by using the superposition principle and suitable extensions of the initial and boundary data.
\\[2mm]
\textit{Decomposition into simpler problems.}
Let $U_0\in H^s(\mathbb R_x^2)$ and $F(t)\in H^s(\mathbb R_x^2)$ be, respectively, extensions of the initial datum $u_0 \in H^s(\mathbb R_{x_1}\times \mathbb R_{x_2}^+)$ and the forcing $f(t) \in H^s(\mathbb R_{x_1}\times \mathbb R_{x_2}^+)$  of ibvp~\eqref{2d-rfls-ibvp}  such that
\begin{equation}\label{ic-f-ext}
\left\|U_0\right\|_{H^s(\mathbb R_x^2)} 
\leqslant
2 \left\|u_0\right\|_{H^s(\mathbb R_{x_1}\times \mathbb R_{x_2}^+)},
\quad
\left\|F(t)\right\|_{H^s(\mathbb R_x^2)} 
\leqslant
2 \left\|f(t)\right\|_{H^s(\mathbb R_{x_1}\times \mathbb R_{x_2}^+)}, 
\   t\in [0, T].
\end{equation}
Then, thanks to linearity, we can express the solution $S\big[u_0, g; f\big]$ of  ibvp~\eqref{2d-rfls-ibvp} in the form
\begin{equation}\label{2d-rnls-decomp}
S\big[u_0, g; f\big]
=
S\big[U_0; 0\big]\big|_{x_2>0} + S\big[0; F\big]\big|_{x_2>0}
+ S\big[0, \psi_1; 0\big]
+ S\big[0, \psi_2; 0\big]
\end{equation}
where, for $U_0$ and $F$ as chosen above, 
$S\big[U_0; 0\big]$ solves the homogeneous linear ivp \eqref{2d-rls-ivp}, 
$S\big[0; F\big]$ solves the forced ivp with zero data \eqref{2d-rfls-ivp}
and, for boundary data $\psi_1$ and $\psi_2$ given by
\begin{equation}\label{psi-chi-def}
\begin{aligned}
\psi_1(x_1, t) &= g(x_1, t) - \left(\p_{x_2} S\big[U_0; 0\big] + \gamma    S\big[U_0; 0\big]\right) (x_1, 0, t),
\\
\psi_2(x_1, t) &= -\left(\p_{x_2} S\big[0; F\big] + \gamma S\big[0; F\big]\right)(x_1, 0, t),
\end{aligned}
\end{equation}
and $S\big[0, \psi_1; 0\big]$ and $S\big[0, \psi_2; 0\big]$  solve the \textit{pure linear ibvp} 
\begin{equation}\label{plibvp}
\begin{aligned}
&iu_t+u_{x_1x_1}+u_{x_2x_2} = 0, \quad (x_1, x_2\in \mathbb R \times \mathbb R^+, \ t \in (0, T),
\\
&u(x_1, x_2, 0)= 0, 
\\
&(u_{x_2} + \gamma u)(x_1, 0, t) = \psi(x_1, t),
\end{aligned}
\end{equation}
with $\psi=\psi_1$ and $\psi=\psi_2$ respectively.
Note that, thanks to Theorems \ref{2d-rls-ivp-t} and \ref{2d-rfls-ivp-t} and the extension inequalities \eqref{ic-f-ext}, the boundary data $\psi_1$ and $\psi_2$ belong to the space $B_T^s$ with  
\begin{equation}\label{psi-est-r}
\left\|\psi_1\right\|_{B_T^s}
\lesssim
\left\|u_0\right\|_{H^s(\mathbb R_{x_1}\times \mathbb R_{x_2}^+)} + \left\|g\right\|_{B_T^s}, 
\quad
\left\|\psi_2\right\|_{B_T^s}
\lesssim
\sqrt T \sup_{t\in [0, T]} \left\|f(t)\right\|_{H^s(\mathbb R_{x_1}\times \mathbb R_{x_2}^+)}.
\end{equation}
In view of  \eqref{2d-rls-ivp-se}, \eqref{2d-rfls-ivp-se} and \eqref{ic-f-ext}, the decomposition \eqref{2d-rnls-decomp}  implies
\begin{align}\label{cr-temp11}
\sup_{t\in [0, T]} \left\|S\big[u_0, g; f\big](t)\right\|_{H^s(\mathbb R_{x_1}\times \mathbb R_{x_2}^+)}
&\leqslant
\left\|u_0\right\|_{H^s(\mathbb R_{x_1}\times \mathbb R_{x_2}^+)}
+
T \sup_{t\in [0, T]} \left\|f(t)\right\|_{H^s(\mathbb R_{x_1}\times \mathbb R_{x_2}^+)}
\nn\\
&\hspace*{-4cm}
+
\sup_{t\in [0, T]} \left\|S\big[0, \psi_1; 0\big](t)\right\|_{H^s(\mathbb R_{x_1}\times \mathbb R_{x_2}^+)}
+
\sup_{t\in [0, T]} \left\|S\big[0, \psi_2; 0\big](t)\right\|_{H^s(\mathbb R_{x_1}\times \mathbb R_{x_2}^+)}.
\end{align}
Next, by relating the pure linear ibvp \eqref{plibvp} with the \textit{reduced} pure linear ibvp \eqref{2d-rnls-pure}, we will deduce estimates for the Hadamard norms of $S\big[0, \psi_1; 0\big]$ and $S\big[0, \psi_2; 0\big]$. 
\\[2mm]
\textit{Extension of the boundary data.}
Given $\psi \in B_T^s$, define $h$ via its $x_1$-Fourier transform as
\begin{equation}\label{h-ext-def}
\widehat h^{x_1}(k_1, t) 
=
\left\{\begin{array}{ll}  \widehat \psi^{x_1}(k_1, t), &t\in (0, T), \\ 0, &t\in [0, T]^c, \end{array}\right. 
\quad \text{a.e. } k_1 \in \mathbb R.
\end{equation}
Note that $\text{supp}(h) \subset \mathbb R_{x_1} \times (0, 2)$ since for the purpose of local well-posedness we take $T<1$. 
Also,  as we show below,  $h \in  X^{0, \frac{2s-1}{4}} \cap X^{s, -\frac 14}$. 
Indeed, for each $k_1 \in \mathbb R$ let $\varphi(t)  = e^{ik_1^2t} \, \widehat \psi^{x_1}(k_1, t)$ and observe that $\psi \in B_T^s$ implies  $\varphi \in H^{\frac{2s-1}{4}}(0, T)$ for a.e. $k_1\in\mathbb R$. 
Then, by Theorem~11.4 of \cite{lm1972}, the extension $\Phi_0$ of $\varphi$ by zero outside $(0, T)$  satisfies
\begin{equation}\label{ext-ineq-low}
\left\| \Phi_0 \right\|_{H^{\frac{2s-1}{4}}(\mathbb R_t)}
\leqslant
c_s \left\| \varphi \right\|_{H^{\frac{2s-1}{4}}(0, T)},
\end{equation}
where $c_s = c(s)$ is independent of $\varphi$.
Furthermore, since $\varphi \in H^{-\frac 14}(0, T)$, there exists an extension $\Phi_1 \in H^{-\frac 14}(\mathbb R_t)$ such that 
$$
\left\| \Phi_1 \right\|_{H^{-\frac 14}(\mathbb R_t)} 
\leqslant 
2 \left\| \varphi \right\|_{H^{-\frac 14}(0, T)}.
$$
Then, noting that $\chi_{(0, T)} \Phi_1 = \Phi_0$ and employing Proposition 3.5 of \cite{jk1995} (see also Lemma 4.2 in~\cite{h2006}), we infer
\begin{equation}\label{m-chi-ineq-3}
\left\| \Phi_0 \right\|_{H^{-\frac 14}(\mathbb R_t)}
=
\left\| \chi_{(0, T)} \Phi_1 \right\|_{H^{-\frac 14}(\mathbb R_t)}
\leqslant
c \left\| \Phi_1 \right\|_{H^{-\frac 14}(\mathbb R_t)}
\leqslant 
2c \left\| \varphi \right\|_{H^{-\frac 14}(0, T)}
\end{equation}
for some universal constant $c$.
Therefore, since  $\widehat h^{x_1}(k_1, t) = e^{-ik_1^2t}\, \Phi_0(k_1, t)$, by the definition of the Bourgain norms \eqref{xsb-def} we find
\begin{align}\label{g-ineq}
\left\|h\right\|_{X^{0, \frac{2s-1}{4}}}^2
&=
\int_{k_1\in\mathbb R} \left\|\Phi_0(k_1)\right\|_{H^{\frac{2s-1}{4}}(\mathbb R_t)}^2 dk_1
\lesssim
\int_{k_1\in\mathbb R} \left\|\varphi(k_1)\right\|_{H^{\frac{2s-1}{4}}(0, T)}^2 dk_1
= \left\|\psi\right\|_{X_T^{0, \frac{2s-1}{4}}}^2,
\nn\\
\left\|h\right\|_{X^{s, -\frac 14}}^2
&=
\int_{k_1\in\mathbb R} \big(1+k_1^2\big)^s \left\|\Phi_0(k_1)\right\|_{H^{-\frac 14}(\mathbb R_t)}^2 dk_1
\nn\\
&\lesssim
\int_{k_1\in\mathbb R} \big(1+k_1^2\big)^s \left\|\varphi(k_1)\right\|_{H^{-\frac 14}(0, T)}^2 dk_1
= \left\|\psi\right\|_{X_T^{s, -\frac 14}}^2,
\end{align}
thereby concluding that $h \in X^{0, \frac{2s-1}{4}} \cap X^{s, -\frac 14}$.

\vskip 2mm
\noindent
\textbf{Proof of Theorem \ref{2d-rfls-t}.}
The extension $h$ of $\psi$ defined by \eqref{h-ext-def} meets all of the requirements for serving as boundary datum in the reduced pure linear ibvp \eqref{2d-rnls-pure}. Thus, we employ estimate~\eqref{2d-rnls-pure-se} along with inequalities \eqref{g-ineq} and \eqref{psi-est-r} to deduce, for $\psi = \psi_1$ and $\psi = \psi_2$ respectively,
\begin{equation}
\begin{aligned}
&\sup_{t\in [0, T]} \left\|S\big[0, \psi_1; 0\big](t)\right\|_{H^s(\mathbb R_{x_1}\times \mathbb R_{x_2}^+)}
\lesssim
\left\|u_0\right\|_{H^s(\mathbb R_{x_1}\times \mathbb R_{x_2}^+)} + \left\|g\right\|_{B_T^s},
\\
&\sup_{t\in [0, T]} \left\|S\big[0, \psi_2; 0\big](t)\right\|_{H^s(\mathbb R_{x_1}\times \mathbb R_{x_2}^+)}
\lesssim
\sqrt T
\sup_{t\in [0, T]}
\left\|f(t)\right\|_{H^s(\mathbb R_{x_1}\times \mathbb R_{x_2}^+)}.
\end{aligned}
\end{equation}
This completes the estimation of the two pure linear ibvps  in \eqref{cr-temp11} and hence implies the desired estimate~\eqref{2d-rfls-se-r} of Theorem \ref{2d-rfls-t} for $\gamma \neq 0$. See Section \ref{neumann-s} for the modifications needed when $\gamma=0$.~\hfill $\square$

\vskip 2mm
\noindent
\textbf{Proof of Theorem \ref{2d-rnls-t}.}
Along the lines of the argument presented in \cite{hm2020}, the forced linear ibvp estimate \eqref{2d-rfls-se-r} can be combined with the algebra property in $H^s(\mathbb R_{x_1} \times \mathbb R_{x_2}^+)$ (which is valid for $s>1$) to show that the iteration map 
$\Phi: u \mapsto S\big[u_0, g; \pm |u|^{\alpha-1}u\big]$
is a contraction in a ball inside the Hadamard space $C([0, T^*]; H^s(\mathbb R_{x_1} \times \mathbb R_{x_2}^+))$ for $1<s<\frac 32$ and lifespan $T^*$ given by \eqref{lifespan}. This amounts to local Hadamard well-posedness   for the NLS ibvp \eqref{2d-rnls-ibvp} as stated in Theorem \ref{2d-rnls-t}.~\hfill $\square$

\section{The Neumann Problem}
\label{neumann-s}

In this section, we provide the modifications needed in the case of Neumann data ($\gamma=0$). We begin with the reduced pure linear ibvp and then  proceed to the proof of Theorem \ref{2d-rfls-t}  for the forced linear ibvp when $\gamma=0$. We note that the proof of Theorem \ref{2d-rnls-t} for the nonlinear problem does not require any modifications from the case $\gamma \neq 0$, since it is based solely on Theorem \ref{2d-rfls-t}, which (as we shall show below) holds  for both $\gamma \neq 0$ and $\gamma = 0$. 
\\[2mm]
\textit{The reduced pure linear ibvp.}
In the Neumann case $\gamma=0$, the reduced pure linear ibvp \eqref{2d-rnls-pure} becomes
\begin{equation}\label{2d-rnls-pure-nn}
\begin{aligned}
&iv_t+v_{x_1x_1}+v_{x_2x_2} = 0, \quad (x_1, x_2)\in \mathbb R \times \mathbb R^+, \ t\in (0, 2),
\\
&v(x_1, x_2, 0)= 0, 
\\
&v_{x_2}(x_1, 0, t) = g(x_1, t), \quad \text{supp}(g)\subset \mathbb R_{x_1}\times (0, 2).
\end{aligned}
\end{equation}
For this problem, the unified transform solution formula \eqref{2d-rfls-sol-T} takes the simple form 
\begin{equation}\label{2d-nls-pure-sol-nn}
v(x_1, x_2, t)
=
-\frac{i}{(2\pi)^2}
\int_{k_1\in\mathbb R}\int_{k_2\in\p D}
e^{ik_1x_1+ik_2x_2-i(k_1^2+k_2^2)t}
\, 2 \, \widehat g(k_1, -k_1^2-k_2^2)\, dk_2dk_1,
\end{equation}
where the contour $\p D$ is the boundary of the first quadrant of the complex $k_2$-plane as shown in Figure~\ref{2d-rnls-dplus-f} and $\widehat g$ denotes the Fourier transform \eqref{ghat-def}.
Using formula \eqref{2d-nls-pure-sol-nn}, we will prove the following Neumann analogue of Theorem \ref{2d-rnls-pure-t}.
\begin{theorem}[\textcolor{blue}{Basic linear estimate for the Neumann problem}]
\label{2d-rnls-pure-nn-t}
Let $s\geqslant 0$. Then, the solution $v(x_1, x_2, t)$ of the reduced pure linear ibvp \eqref{2d-rnls-pure-nn}, as given by  formula \eqref{2d-nls-pure-sol-nn}, satisfies the Hadamard estimate
\begin{equation}\label{2d-rnls-pure-se-nn-opt}
\sup_{t\in [0, 2]} \left\| v(t) \right\|_{H^s(\mathbb R_{x_1}\times \mathbb R_{x_2}^+)}
\leqslant
c_s
\big(
\left\|g\right\|_{\dot X^{0, \frac{2s-1}{4}}} 
+
\left\|g\right\|_{\dot X^{s, -\frac 14}}
\big).
\end{equation}
\end{theorem}

In the above theorem,   the homogeneous-in-time Bourgain space $\dot X^{\sigma, b}$  is defined via the norm
\begin{equation}\nn
\left\|g\right\|_{\dot X^{\sigma, b}}^2
=
\int_{k_1\in\mathbb R} 
\!\! \big(1+k_1^2\big)^\sigma 
\big|\tau+k_1^2\big|^{2b}
\left|\widehat g(k_1, \tau)\right|^2 dk_1
=
\int_{k_1\in\mathbb R} 
\!\! \big(1+k_1^2\big)^\sigma 
\left\|e^{ik_1^2t} \, \widehat g^{x_1}(k_1, t)\right\|_{\dot H^b(\mathbb R_t)}^2 \! dk_1,
\end{equation}
where  $\dot H^b$ denotes the homogeneous Sobolev space. Importantly, we remark that for $s\geqslant \frac 12$ the space $\dot X^{0, \frac{2s-1}{4}}$ \textit{can} be replaced in \eqref{2d-rnls-pure-se-nn-opt} by the nonhomogeneous space $X^{0, \frac{2s-1}{4}}$. However, this is \textit{not} possible for the space $\dot X^{s, -\frac 14}$, since the homogeneous weight is now independent of $s$.

\begin{proof}
Parametrizing  $\p D$ along the positive imaginary and real axes, we  write $v =  v_1 + v_2$
 with 
\begin{align*}
v_1(x_1, x_2, t)
&=
-\frac{1}{(2\pi)^2}
\int_{k_1\in\mathbb R}
\int_{k_2=0}^\infty
e^{ik_1x_1-k_2x_2-i(k_1^2-k_2^2)t}
\,   2 \, \widehat g(k_1, -k_1^2+k_2^2) \, dk_2 dk_1,
\\
v_2(x_1, x_2, t)
&=
-\frac{i}{(2\pi)^2}
\int_{k_1\in\mathbb R}\int_{k_2=0}^\infty
e^{ik_1 x_1+ik_2x_2-i(k_1^2+k_2^2)t}
\,  2 \, \widehat g( k_1, -k_1^2-k_2^2) \, dk_2 dk_1.
\end{align*}
These expressions are similar to \eqref{2d-v1-def} and \eqref{2d-v2-def} for the Robin problem but without the fractions 
$\frac{k_2}{k_2-\gamma}$ and $\frac{k_2}{k_2-i\gamma}$ respectively. Thus, with the exception of  inequality \eqref{1/gamma}, which was the key for transitioning from homogeneous to nonhomogeneous weights in the Robin case, we can  follow the same steps as in the proof of Theorem \ref{2d-rnls-pure-t} to obtain the following analogues of \eqref{2d-nls-v2-se-r} and \eqref{2d-nls-v1-se-r}  
\begin{align*}
\left\|v_2(t)\right\|_{H^s(\mathbb R_{x_1}\times\mathbb R_{x_2}^+)}^2 
&\lesssim
\int_{k_1\in\mathbb R}
\int_{\tau=-\infty}^{-k_1^2}
\big(1+k_1^2\big)^s   
\big|\tau+k_1^2\big|^{-\frac 12}
\big| \widehat g(k_1, \tau)\big|^2   d\tau dk_1
\nn\\
&\quad
+
\int_{k_1\in\mathbb R}
\int_{\tau=-\infty}^{-k_1^2}
\big|\tau+k_1^2\big|^{s-\frac 12}
\big| \widehat g(k_1, \tau)\big|^2 d\tau  dk_1, \quad s\in \mathbb R,
\\
\left\| v_1(t) \right\|_{H^s(\mathbb R_{x_1}\times \mathbb R_{x_2}^+)}^2
&\lesssim
\int_{k_1\in\mathbb R} 
\int_{\tau=-k_1^2}^\infty
\big(1+k_1^2\big)^s
\left(\tau+k_1^2\right)^{-\frac 12}
\left|
\widehat g(k_1, \tau)
\right|^2
d\tau dk_1  
\nn\\
&\quad
+
\int_{k_1\in\mathbb R} 
\int_{\tau=-k_1^2}^\infty
\left(\tau+k_1^2\right)^{s-\frac 12}
\left|
\widehat g(k_1, \tau)
\right|^2
d\tau dk_1, \quad s\geqslant 0.
\end{align*}
These two estimates combine to imply the desired estimate \eqref{2d-rnls-pure-se-nn-opt}.
\end{proof}

\noindent
\textbf{Proof of Theorem \ref{2d-rfls-t} when $\gamma=0$.}
Having established Theorem \ref{2d-rnls-pure-nn-t} for the reduced pure linear ibvp in the Neumann case, we shall now  combine this result with inequality \eqref{cr-temp11} in order to establish estimate~\eqref{2d-rfls-se-r} for the forced linear ibvp \eqref{2d-rfls-ibvp}  when $\gamma=0$ and $1\leqslant s < \frac 32$. Like in the Robin case, this step requires us to define an appropriate extension of the boundary data \eqref{psi-chi-def}, which for $\gamma = 0$ are given by
\begin{equation}\label{psi-chi-def-nn}
\psi_1(x_1, t) = g(x_1, t) -  \p_{x_2} S\big[U_0; 0\big](x_1, 0, t),
\quad
\psi_2(x_1, t) =  - \p_{x_2} S\big[0; F\big](x_1, 0, t).
\end{equation}
Indeed, the remaining two components $S\big[U_0; 0\big]$ and  $S\big[0; F\big]$ that appear in \eqref{2d-rnls-decomp} concern the homogeneous and forced linear ivp respectively and have been estimated in Theorems~\ref{2d-rls-ivp-t} and~\ref{2d-rfls-ivp-t}.

Denoting $\psi_1$ and $\psi_2$ simply by $\psi \in B_T^s$, for each $k_1 \in \mathbb R$ we let  $\varphi(t)  = e^{ik_1^2t} \, \widehat \psi^{x_1}(k_1, t)$ and denote the extension of this function outside $(0, T)$ by $\Phi_0(t)$. 
Observe that, since $\psi \in B_T^s$, we have $\varphi \in H^{\frac{2s-1}{4}}(0, T)$ for a.e. $k_1 \in \mathbb R$. 
Next, define the global function $h(x_1, t)$ via its $x_1$-Fourier transform as
\begin{equation}\label{h-Phi-nn}
\widehat h^{x_1}(k_1, t) 
=
\theta(t) \, e^{-ik_1^2t} \, \Phi_0(k_1, t), \quad \text{a.e. } k_1 \in \mathbb R, \ t\in \mathbb R,
\end{equation}
where $\theta\in C_0^\infty[-1, 3]$ is equal to $1$ on $[0, 2]$ and decays to zero smoothly outside $[0, 2]$. 
Note that  this definition implies $\text{supp}(h) \subset \mathbb R_{x_1} \times (0, 2)$ and $h(t) = \psi(t)$ for $t\in (0, T)$.
Also, as we show next, $h \in X^{0, \frac{2s-1}{4}} \cap X^{s, -\frac 14}$.
Indeed, the fact $\left\| \theta f \right\|_{H^b(\mathbb R)} \leqslant c(\theta, b) \left\| f \right\|_{H^b(\mathbb R)}$, $b\in\mathbb R$, combined with inequalities~\eqref{ext-ineq-low} and~\eqref{m-chi-ineq-3} yields
$$
\left\|\theta \Phi_0(k_1)\right\|_{H^{\frac{2s-1}{4}}(\mathbb R_t)}
\leqslant
c_{s, \theta}
\left\|\varphi(k_1)\right\|_{H^{\frac{2s-1}{4}}(0, T)},
\quad
\left\|\theta \Phi_0(k_1)\right\|_{H^{-\frac 14}(\mathbb R_t)}
\leqslant
c_{s, \theta}
\left\|\varphi(k_1)\right\|_{H^{-\frac14}(0, T)}.
$$
Thus, by the definition of the Bourgain norms, we find
\begin{align}
&\left\|h\right\|_{X^{0, \frac{2s-1}{4}}}^2
=
\int_{k_1\in\mathbb R} \left\|\theta \Phi_0(k_1)\right\|_{H^{\frac{2s-1}{4}}(\mathbb R_t)}^2 dk_1
\lesssim
\int_{k_1\in\mathbb R} \left\|\varphi(k_1)\right\|_{H^{\frac{2s-1}{4}}(0, T)}^2 dk_1
= \left\|\psi\right\|_{X_T^{0, \frac{2s-1}{4}}}^2,
\nn\\
&\left\|h\right\|_{X^{s, -\frac 14}}^2
=
\int_{k_1\in\mathbb R} \big(1+k_1^2\big)^s \left\|\theta \Phi_0(k_1)\right\|_{H^{-\frac 14}(\mathbb R_t)}^2 dk_1
\lesssim
\int_{k_1\in\mathbb R} \big(1+k_1^2\big)^s \left\|\varphi(k_1)\right\|_{H^{-\frac 14}(0, T)}^2 dk_1
= \left\|\psi\right\|_{X_T^{s, -\frac 14}}^2,
\nn
\end{align}
thereby deducing that $h \in X^{0, \frac{2s-1}{4}} \cap X^{s, -\frac 14}$.
Therefore, $h$ can play the role of the boundary datum in the 
reduced pure linear ibvp \eqref{2d-rnls-pure-nn} and so, for $\psi = \psi_1$,   estimate \eqref{2d-rnls-pure-se-nn-opt} yields
\begin{equation}\label{ext-temp1}
\hspace*{-3mm}
\sup_{t\in [0, T]} \left\|S\big[0, \psi_1; 0\big](t)\right\|_{H^s(\mathbb R_{x_1}\times \mathbb R_{x_2}^+)}
\lesssim
\left\| h \right\|_{X^{0, \frac{2s-1}{4}}} + \left\| h \right\|_{\dot X^{s, -\frac 14}}
\lesssim
\left\| \psi_1 \right\|_{X_T^{0, \frac{2s-1}{4}}} + \left\| h \right\|_{\dot X^{s, -\frac 14}}
\end{equation}
after replacing the norm of $\dot X^{0, \frac{2s-1}{4}}$ with the one of $X^{0, \frac{2s-1}{4}}$ (since $s > \frac 12$).
In addition, according to Lemma 2.8 in \cite{ck2002} (see also Lemma 4.1 in \cite{h2006}), for $0\leqslant b < \frac 12$ and any $\theta \in C_0^\infty(\mathbb R)$ we have
$
\left\| \theta f \right\|_{\dot H^{-b}(\mathbb R)}
\leqslant
c(\theta, b) \left\| f \right\|_{H^{-b}(\mathbb R)}.
$
Using this result for $b=\frac 14$, $\theta$ as above and $f = \Phi_0$,  we obtain
$
\left\| \theta \Phi_0(k_1) \right\|_{\dot H^{-\frac 14}(\mathbb R_t)}
\lesssim
\left\| \Phi_0(k_1) \right\|_{H^{-\frac 14}(\mathbb R_t)}
\lesssim
\left\| \varphi(k_1) \right\|_{H^{-\frac 14}(0, T)}$,
where for the last step we have employed inequality \eqref{m-chi-ineq-3}. Hence,
$$
\left\| h \right\|_{\dot X^{s, -\frac 14}}^2
=
\int_{k_1\in\mathbb R} \big(1+k_1^2\big)^s \left\|\theta \Phi_0(k_1)\right\|_{\dot H^{-\frac 14}(\mathbb R_t)}^2 dk_1
\lesssim
\int_{k_1\in\mathbb R} \big(1+k_1^2\big)^s \left\|\varphi(k_1)\right\|_{H^{-\frac 14}(0, T)}^2 dk_1
=
\left\| \psi_1 \right\|_{X_T^{s, -\frac 14}}^2
$$
and, combining \eqref{ext-temp1} with \eqref{psi-est-r} for $\psi_1$, we find
\begin{equation}\nn
\sup_{t\in [0, T]} \left\|S\big[0, \psi_1; 0\big](t)\right\|_{H^s(\mathbb R_{x_1}\times \mathbb R_{x_2}^+)}
\lesssim
\left\|g\right\|_{B_T^s}
+
\left\|u_0\right\|_{H^s(\mathbb R_{x_1}\times \mathbb R_{x_2}^+)}.
\end{equation}
Along the same lines, we also have
\begin{equation}\nn
\sup_{t\in [0, T]} \left\|S\big[0, \psi_2; 0\big](t)\right\|_{H^s(\mathbb R_{x_1}\times \mathbb R_{x_2}^+)}
\lesssim
\sqrt T
\sup_{t\in [0, T]}
\left\|f(t)\right\|_{H^s(\mathbb R_{x_1}\times \mathbb R_{x_2}^+)}.
\end{equation}
These estimates along with inequality \eqref{cr-temp11} imply the desired estimate \eqref{2d-rfls-se-r}, completing the proof of Theorem \ref{2d-rfls-t} in the Neumann case $\gamma=0$.~\hfill $\square$

%
%
%
%
%
%
%
%
\section{Solution of the linear ibvp via the unified transform of Fokas}
\label{2d-rnls-utm-s}

We provide the derivation of the unified transform formula \eqref{2d-rfls-sol-T} for the forced linear Schr\"odinger ibvp \eqref{2d-rfls-ibvp} under the assumption of sufficiently smooth initial and boundary values. 
Taking the half-plane Fourier transform \eqref{ft-hp-def} of the forced linear Schr\"odinger equation and using the Robin boundary condition along with the notation  $u(x_1, 0, t) = g_0(x_1, t)$, we find
\begin{equation}
\widehat u_t  + i\big(k_1^2+k_2^2\big)  \widehat u 
=
\int_{x_1\in\mathbb R} e^{-ik_1x_1} \left[\left(k_2 + i\gamma\right) g_0(x_1, t) - ig(x_1, t) \right] dx_1 
-
i\widehat f,
\quad k_1 \in \mathbb R, \ \text{Im}(k_2) \leqslant 0.
\nn
\end{equation}
We note that the domain can be extended to the lower half of the complex $k_2$-plane thanks to the fact that $x_2 > 0$. 
Integrating this expression with respect to $t$ yields the so-called \textit{global relation}
\begin{align}\label{2d-rfls-gr}
e^{i(k_1^2+k_2^2)t}\, \widehat u (k_1, k_2, t)
&=
\widehat u_0(k_1, k_2)
+
\left(k_2+i\gamma\right) \widetilde g_0(k_1, k_1^2+k_2^2, t) - i \,\widetilde g(k_1, k_1^2+k_2^2, t)
\nn\\
&\quad
-i\int_{t'=0}^t e^{i(k_1^2+k_2^2)t'} \widehat f(k_1, k_2, t') dt'
\end{align}
with $\widetilde g_0$ and $\widetilde g$ defined according to \eqref{gtil-def}.
Thus, by the inverse Fourier transform we obtain
\begin{align}\label{2d-rfls-ir}
&
u(x_1, x_2, t)
=
\frac{1}{(2\pi)^2} \int_{k_1\in\mathbb R}\int_{k_2\in\mathbb R} e^{ik_1x_1+ik_2x_2-i(k_1^2+k_2^2)t}\, \widehat u_0(k_1, k_2) dk_2 dk_1
\\
& 
-
\frac{i}{(2\pi)^2}\int_{k_1\in\mathbb R}\int_{k_2\in\mathbb R} e^{ik_1x_1+ik_2x_2-i(k_1^2+k_2^2)t} \int_{t'=0}^t e^{i(k_1^2+k_2^2) t'} \widehat f(k_1, k_2, t') dt' dk_2 dk_1
\nn\\
&
+
\frac{1}{(2\pi)^2} \int_{k_1\in\mathbb R}\int_{k_2\in\mathbb R} e^{ik_1x_1+ik_2x_2-i(k_1^2+k_2^2)t} \Big[\left(k_2+i\gamma\right)\widetilde g_0(k_1, k_1^2+k_2^2, t) 
- i \, \widetilde g(k_1, k_1^2+k_2^2,t) \Big] dk_2 dk_1.
\nn
\end{align}

This is an integral representation for the solution but \textit{not} an explicit  formula  since  $\widetilde g_0$ involves the unknown Dirichlet boundary value $g_0$. However, it is possible to eliminate $\widetilde g_0$   from \eqref{2d-rfls-ir} in favor of known quantities.
To accomplish this, we begin by observing that for $x_2\geqslant 0$ and $t\geqslant t'$ the exponential $e^{ik_2x_2-i k_2^2(t-t')}$ is bounded for  $k_2 \in \left\{\text{Im}(k_2)\geqslant 0\right\}\setminus D$, where $D$ here denotes the first quadrant of the complex $k_2$-plane (see Figure \ref{2d-rnls-dplus-f}). Thus, exploiting the analyticity of the half-plane Fourier transform as well as of the transforms $\widetilde g_0$ and $\widetilde g$ for all $k_2\in\mathbb C$ (which follows via a Paley-Wiener theorem), we apply Cauchy's theorem in the second quadrant of the complex $k_2$-plane to deform the contour of the $k_2$ integral in the last term of \eqref{2d-rfls-ir} from $\mathbb R$ to  the positively oriented boundary $\p D$ of  $D$ (see Figure \ref{2d-rnls-dplus-f}). This deformation is possible thanks to the fact that, for the quartercircle 
$
\gamma_\rho^+
=
\big\{\rho e^{i\theta}:  \tfrac \pi 2\leqslant \theta \leqslant \pi\big\}
$,
we can show along the lines of \cite{hm2020} that
$\lim_{\rho \to \infty}
 \int_{k_2\in \gamma_\rho^+} e^{ik_2x_2-ik_2^2t} \big[\left(k_2+i\gamma\right)\widetilde g_0(k_1, k_1^2+k_2^2, t) - i \, \widetilde g(k_1, k_1^2+k_2^2,t) \big] dk_2
= 0$. 
Thus, \eqref{2d-rfls-ir} becomes
\begin{align}
\label{2d-rfls-ir-def}
&u(x_1, x_2, t)
=
\frac{1}{(2\pi)^2} \int_{k_1\in\mathbb R}\int_{k_2\in\mathbb R} e^{ik_1x_1+ik_2x_2-i(k_1^2+k_2^2)t}\, \widehat u_0(k_1, k_2) dk_2 dk_1
\\
& 
-
\frac{i}{(2\pi)^2}\int_{k_1\in\mathbb R}\int_{k_2\in\mathbb R} e^{ik_1x_1+ik_2x_2-i(k_1^2+k_2^2)t} \int_{t'=0}^t e^{i(k_1^2+k_2^2) t'} \widehat f(k_1, k_2, t') dt' dk_2 dk_1
\nn\\
&
+
\frac{1}{(2\pi)^2} \int_{k_1\in\mathbb R}\int_{k_2\in\p D} e^{ik_1x_1+ik_2x_2-i(k_1^2+k_2^2)t} 
\Big[\left(k_2+i\gamma\right)\widetilde g_0(k_1, k_1^2+k_2^2, t) 
- i \, \widetilde g(k_1, k_1^2+k_2^2,t) \Big] dk_2 dk_1.
\nn
\end{align}

Next, note that under the transformation $k_2\mapsto -k_2$  the global relation \eqref{2d-rfls-gr} yields the identity
\begin{align}\label{2d-rfls-gr-}
e^{i(k_1^2+k_2^2)t}\, \widehat u(k_1, -k_2, t)
&=
\widehat u_0(k_1, -k_2)
+
\left(-k_2+i\gamma\right) \widetilde g_0(k_1, k_1^2+k_2^2, t) - i \, \widetilde g(k_1, k_1^2+k_2^2, t)
\nn\\
&\quad
-i\int_{t'=0}^t e^{i(k_1^2+k_2^2)t'} \widehat f(k_1, -k_2, t') dt', \quad k_1 \in \mathbb R, \ \text{Im}(k_2) \geqslant 0.
\end{align}
If $\gamma < 0$, then $k_2-i\gamma \neq 0$ for all $k_2 \in \p D$, so we can use  \eqref{2d-rfls-gr-} to substitute for $\widetilde g_0$ in \eqref{2d-rfls-ir-def} and obtain
\begin{align}\label{2d-rfls-sol-t}
&\quad
u(x_1, x_2, t)
=
\frac{1}{(2\pi)^2} \int_{k_1\in\mathbb R} \int_{k_2\in\mathbb R} e^{ik_1x_1+ik_2x_2-i(k_1^2+k_2^2)t} \, \widehat u_0(k_1, k_2) dk_2 dk_1
\nn\\
&
\quad
+\frac{1}{(2\pi)^2} \int_{k_1\in\mathbb R} \int_{k_2\in\p D}
e^{ik_1x_1+ik_2x_2-i(k_1^2+k_2^2)t} \, \frac{k_2+i\gamma}{k_2-i\gamma} \, \widehat u_0(k_1,-k_2) dk_2 dk_1
\nn\\
&\quad
-\frac{i}{(2\pi)^2}
\int_{k_1\in\mathbb R}
\int_{k_2\in\mathbb R}
e^{ik_1x_1+ik_2x_2-i(k_1^2+k_2^2)t}
\int_{t'=0}^t 
e^{i(k_1^2+k_2^2)t'}\widehat f(k_1, k_2, t')dt'dk_2 dk_1
\nn\\
&
\quad
-\frac{i}{(2\pi)^2}
\int_{k_1\in\mathbb R}
\int_{k_2\in\p D}
e^{ik_1x_1+ik_2x_2-i(k_1^2+k_2^2)t}  \, \frac{k_2+i\gamma}{k_2-i\gamma} \,
\int_{t'=0}^t 
e^{i(k_1^2+k_2^2)t'}\widehat f(k_1, -k_2, t')dt'
dk_2 dk_1
\nn\\
&\quad
-\frac{i}{(2\pi)^2}
\int_{k_1\in\mathbb R}\int_{k_2\in\p D}
e^{ik_1x_1+ik_2x_2-i(k_1^2+k_2^2)t}
 \, \frac{2k_2}{k_2-i\gamma} \, \widetilde g(k_1, k_1^2+k_2^2, t)
dk_2dk_1,
\end{align}
where we have also used the fact that 
$
\int_{k_2\in \p D}e^{ik_2x_2}  \, \frac{k_2+i\gamma}{k_2-i\gamma} \, \widehat u(k_1,-k_2,t)dk_2 = 0
$
by analyticity and exponential decay of the integrand inside $D$. 
Formula \eqref{2d-rfls-sol-t} is true also for $\gamma=0$, since in that case the denominator $k_2-i\gamma$ cancels out.
Moreover, exploiting once again analyticity and exponential decay in $D$, we infer  
$
\int_{k_2\in \p D} e^{ik_2x_2-i k_2^2 t} \int_{t'=t}^T e^{i(k_1^2+k_2^2) t'}  \, \frac{2k_2}{k_2-i\gamma} \int_{y_1\in \mathbb R} e^{-ik_1y_1} \, g (y_1, t') dy_1 dt' dk_2 = 0
$,
which turns \eqref{2d-rfls-sol-t} into the equivalent form \eqref{2d-rfls-sol-T} with $\mathcal C = \p D$.

If $\gamma>0$, then $k_2-i\gamma$ vanishes along the positive imaginary axis, which is part of $\p D$. To avoid crossing this singularity, before using  identity \eqref{2d-rfls-gr-} to solve for $\widetilde g_0$ we locally deform the contour of integration of the last $k_2$ integral in \eqref{2d-rfls-ir} from $\p D$ to the contour $\p \widetilde D$ shown in Figure~\ref{2d-rnls-dplus-f}. Then, proceeding as for $\gamma<0$, we obtain the unified transform formula \eqref{2d-rfls-sol-T}, this time with $\mathcal C = \p \widetilde D$.

%
%
%
%
\vspace*{3mm}
\noindent
{\bf Acknowledgements.} This work was partially supported by a grant from the Simons Foundation (\#524469 to Alex Himonas) and a grant from the National Science Foundation (NSF-DMS 2206270 to D. Mantzavinos).

%
%
%
%
%
%
%

\vspace*{5mm}

\begin{center}
\setlength{\tabcolsep}{5mm}
\begin{tabular}{ll}
A. Alexandrou Himonas  &Dionyssios Mantzavinos
\\
Department of Mathematics &Department of Mathematics \\
University of Notre Dame   &University of Kansas
\\
Notre Dame, IN 46556 &Lawrence, KS 66045 \\
\textit{himonas.1$@$nd.edu} & \textit{mantzavinos@ku.edu}
\end{tabular}
\end{center}

\end{document}